\documentclass[a4,12pt, reqno]{amsart}

\usepackage{amssymb}
\usepackage{amstext}
\usepackage{amsmath}
\usepackage{amscd}
\usepackage{amsthm}
\usepackage{amsfonts}
\usepackage{enumerate}
\usepackage{fancybox}
\usepackage{graphicx}
\usepackage{latexsym}
\usepackage[all,arc]{xy}
\usepackage[usenames]{color}
\usepackage{mathrsfs}
\usepackage{color}
\usepackage{arydshln}
\usepackage{multienum}

\newenvironment{red}
{\relax\color{red}}
{\hspace*{.5ex}\relax}

\newcommand{\ber}{\begin{red}}
\newcommand{\er}{\end{red}}

\renewcommand{\sf}{\mathsf}
\newenvironment{verd}
{\relax\color{magenta}}
{\hspace*{.5ex}\relax}

\newcommand{\bg}{\begin{verd}}
\newcommand{\eg}{\end{verd}}

\theoremstyle{definition}
\newtheorem*{maintheorem}{Main Theorem}
\newtheorem{theorem}{Theorem}[section]
\newtheorem{R-theorem}[theorem]{The Riedtmann structure theorem}
\newtheorem{corollary}[theorem]{Corollary}
\newtheorem{lemma}[theorem]{Lemma}
\newtheorem{proposition}[theorem]{Proposition}
\newtheorem{definition}[theorem]{Definition}
\newtheorem{remark}[theorem]{Remark}

\def \para{\refstepcounter{theorem} \par\medskip\noindent
                \textbf{\thetheorem .} }

\theoremstyle{remark}

\numberwithin{equation}{theorem}

\renewcommand{\O}{\mathcal{O}}

\newcommand{\Z}{{\mathbb{Z}}}

\renewcommand{\k}{{\mathbf{k}}}

 \def\hsymb#1{\mbox{\strut\rlap{\smash{\Huge$#1$}}\quad}}
\begin{document}
\title[On the Heller lattices of Kronecker algebra]{On periodic stable Auslander--Reiten components containing Heller lattices over the symmetric Kronecker algebra}
\date{}

\author[Kengo Miyamoto]{Kengo Miyamoto}
\address{{Ibaraki University, 4-12-1 Nakanarusawa, Hitachi, Ibaraki, 316-8511, Japan.}}
\email{{}kengo.miyamoto.uz63@vc.ibaraki.ac.jp}

\thanks{The author is a Research Fellow of Japan Society for the Promotion of Science (JSPS). 
This work was supported by Japan Society for the Promotion of Science KAKENHI 18J10561 {and 21K17702.}}

\begin{abstract} Let $\mathcal{O}$ be a complete discrete valuation ring, $\mathcal{K}$ its quotient field, and $A$ the symmetric Kronecker algebra over $\mathcal{O}$. We consider the full subcategory of the category of $A$-lattices whose objects are $A$-lattices $M$ such that $M\otimes_{\mathcal{O}}\mathcal{K}$ is projective $A\otimes_{\mathcal{O}}\mathcal{K}$-modules.  In this paper, we study Heller lattices of indecomposable periodic modules over {$A$}. 
As a main result,  we determine the shapes of stable Auslander--Reiten components containing Heller lattices of indecomposable periodic modules over  {$A$}. 
  \end{abstract}
\maketitle
\tableofcontents
\section*{Introduction}
In representation theory of algebras, we often use Auslander--Reiten theory to {analyze} various additive categories and prove many important combinatorial and homological properties with the help of the theory, for example, \cite{ARS, ASS, H, I3, Li2, Y}. 
In the case for the category of lattices over an order, see \cite{A2, Bu, I2, K2, RoS}. 

Let $\mathcal{O}$ be a complete discrete valuation ring with a uniformizer $\varepsilon$ and $\mathcal{K}$ the quotient field of $\mathcal{O}$. An $\mathcal{O}$-algebra $A$ is an \textit{$\mathcal{O}$-order} if $A$ is free of finite rank as an $\mathcal{O}$-module. We write $\overline{A}$ for the induced algebra $A\otimes_{\mathcal{O}}(\mathcal{O}/\varepsilon\mathcal{O})$. For an $\mathcal{O}$-order $A$, a right $A$-module $M$ is called an \textit{$A$-lattice} if $M$ is free of finite rank as an $\mathcal{O}$-module. We denote by $\mathsf{latt}$-$A$ the full subcategory of the module category $\mathsf{mod}$-$A$ consisting of $A$-lattices.
According to \cite{A2}, the category $\sf{latt}$-$A$ admits almost split sequences if and only if  $A$ is an isolated singularity, that is, $A\otimes_{\mathcal{O}}\mathcal{K}$ is a semisimple q$\mathcal{K}$-algebra. In this case, one can find some results on the shapes of Auslander--Reiten quivers, for example, \cite{K2, Lu, W}. 

When $A$ is not an isolated singularity, we have to consider a suitable full subcategory of $\sf{latt}$-$A$ which admits almost split sequences. It follows from \cite[Theorem 2.1]{AR} that $M\in\sf{latt}$-$A$ appears at the end term of an almost split sequence if and only if $M$ satisfies the condition $(\natural)$:
\begin{center}
$M\otimes_{\mathcal{O}}\mathcal{K}$ is projective as an $A\otimes_{\mathcal{O}}\mathcal{K}$-module.\quad\quad $(\natural)$
\end{center}
Here, the full subcategory of $\sf{latt}$-$A$ consisting of $A$-lattices which satisfy the condition $(\natural)$ is denoted by $\sf{latt}^{(\natural)}$-$A$. When $A$ is symmetric, that is, $A$ is isomorphic to $\mathrm{Hom}_{\mathcal{O}}(A,\mathcal{O})$ as $(A,A)$-bimodules, the category $\sf{latt}^{(\natural)}$-$A$ admits almost split sequences. Thus, Ariki, Kase, and the author defined the concept of the stable Auslander--Reiten quiver for $\sf{latt}^{(\natural)}$-$A$, and found several conditions to be satisfied on the shapes of stable periodic components with infinitely many vertices by using Riedtmann's structure theorem. As another restriction, the author proved that the tree class of any stable component is one of infinite Dynkin diagrams or Euclidean diagrams when $\overline{A}$ is of finite representation type \cite{M}. 
However, the shapes of stable components of an $\mathcal{O}$-order seems to be largely unknown, and there are only few concrete examples since it is difficult to compute almost split sequences in general.

Let $A$ be a symmetric $\mathcal{O}$-order. To get a new example of stable components for $\mathsf{latt}^{(\natural)}$-$A$, we consider a special kind of $A$-lattices called \textit{Heller lattices}, which is defined to be direct summands of the first syzygy of an indecomposable $\overline{A}$-module as an $A$-module. There are two reasons why we consider Heller lattices. The first reason is that they always belong to $\sf{latt}^{(\natural)}$-$A$. Thus, the category $\sf{latt}^{(\natural)}$-$A$ admits some stable components containing indecomposable Heller lattices. We call such components \textit{Heller components} of $A$. 
Another reason is that Heller lattices of a group algebra play important roles in modular representation theory.
For a $p$-modular system $(\mathcal{K},\mathcal{O},\kappa)$ for a finite group $G$, Heller lattices over $\mathcal{O}G$ were studied by Kawata \cite{K, K2}. It follows from \cite[Theorem 4.4]{K} that Heller lattices over $\mathcal{O}G$ provide us with a certain relationship between almost split sequences for $\mathsf{latt}$-$\mathcal{O}G$ and $\mathsf{mod}$-$\kappa G$, namely he showed that if $0\to A\to B\to Z_M\to 0$ is the almost split sequence ending at an indecomposable Heller lattice $Z_M$ of an indecomposable $\kappa G$-module $M$, then the induced exact sequence
\[ 0\to A\otimes_{\mathcal{O}}\kappa\to B\otimes_{\mathcal{O}}\kappa\to Z_M\otimes_{\mathcal{O}}\kappa\to 0\] 
is the direct sum of the almost split sequence ending at $M$ and a split sequence (see also \cite[Corollary 5.8]{P}).
They motivate us to study Heller lattices when $A$ is an arbitrary symmetric $\mathcal{O}$-order. In \cite{AKM}, we studied Heller lattices over truncated polynomial rings and determined the shapes of stable components containing indecomposable Heller lattices. This is the first example of a stable Auslander--Reiten component containing Heller lattices when $A\otimes_{\mathcal{O}}\mathcal{K}$ is not semisimple. 

In this paper, we consider the symmetric Kronecker algebra $A=\mathcal{O}[X, Y]/(X^2, Y^2)$. Then, the Auslander--Reiten quiver of $\overline{A}$ consists of a unique non-periodic component, which contains the simple $\overline{A}$-module, and infinitely many homogeneous tubes \cite{ARS, ASS}. In \cite{M}, I studied Heller lattices of indecomposable non-periodic $\overline{A}$-modules, and showed that $\sf{latt}^{(\natural)}$-$A$ admits a unique non-periodic Heller component containing them, and it is of the form $\mathbb{Z}A_{\infty}$. In this article, we focus on the remaining Heller lattices, and we will show that they are indecomposable. It is well-known that such homogeneous tubes are classified by the projective line $\mathbb{P}^1(\kappa)$ \cite{ARS, SS}. Hence, Heller lattices of indecomposable periodic $\overline{A}$-modules are parametrized by $\mathbb{Z}_{>0}\times\mathbb{P}^1(\kappa)$. We denote by $Z_m^{\lambda}$ the Heller lattice associated with $(m,\lambda)\in\mathbb{Z}_{>0}\times\mathbb{P}^1(\kappa)$.
The main result is the following.

\begin{maintheorem}[Theorems \ref{main1} and \ref{main2}]
Let $\mathcal{O}$ be a complete discrete valuation ring, $\kappa$ the residue field and $A=\mathcal{O}[X,Y]/(X^{2},Y^{2})$. Suppose that $\kappa$ is algebraically closed. Let $\mathcal{CH}(Z_{m}^{\lambda})$ be the stable Auslander--Reiten component for $\mathsf{latt}^{(\natural)}$-$A$ containing $Z_{m}^{\lambda}$. 
Then, the following statements hold.
\begin{enumerate}
\item Assume that $\sf{Char}(\kappa)=2$, then $\mathcal{CH}(Z_{m}^{\lambda})\simeq \mathbb{Z}A_{\infty}/\langle \tau \rangle$ for all $\lambda\in\mathbb{P}^1(\kappa)$.
\item Assume that $\sf{Char}(\kappa)\neq 2$, then 
\[ \mathcal{CH}(Z_{m}^{\lambda})\simeq \left\{\begin{array}{ll}
\mathbb{Z}A_{\infty}/\langle \tau \rangle &  \text{if $\lambda=0$ or $\infty$,}\\
\mathbb{Z}A_{\infty}/\langle \tau^2 \rangle & \text{otherwise.} \end{array}\right. \] 
\end{enumerate}
Moreover, the Heller lattice $Z_{m}^{\lambda}$ appears on the boundary of $\mathcal{CH}(Z_{m}^{\lambda})$.
\end{maintheorem}

This paper consists of four sections. In Section 1, we recall some notions, including almost split sequences, stable Auslander--Reiten quivers, and some results from \cite{A, AKM, Ri}. 
In Section 2, we give a complete list of Heller lattices of $A = \mathcal{O}[X, Y]/(X^2, Y^2)$ and explain their properties, including the indecomposability and the periodicity/aperiodicity. 
In Section 3, we consider the case $\lambda\neq \infty$ and determine the shape of the stable Auslander--Reiten component containing $Z_m^{\lambda}$. Moreover, we show that every Heller lattice $Z_m^{\lambda}$ appears on the boundary of the Heller component of $A$. Note that, using Riedtmann's structure theorem, it is not necessary to construct all almost split sequences to determine the shape of Heller components of $A$. In fact, we only construct the almost split sequences ending at $Z_m^{\lambda}$. Finally, we consider the case $\lambda=\infty$ in Section 4. 

\section*{Acknowledgement}
My heartfelt appreciation goes to Professor Susumu Ariki (Osaka University), who provided helpful comments and suggestions. I would also like to thank Professor Shigeto Kawata (Nagoya City University), Professor Michihisa Wakui (Kansai University), Professor Ryoichi Kase (Okayama University of Science), and Dr. Qi Wang ({Tsinghua University}) whose meticulous comments were an enormous help to me.


\section{Preliminaries}
Throughout this paper, we use the following conventions.
\begin{enumerate}[(1)]
\item {$\mathcal{O}$ denotes a complete discrete valuation ring with a uniformizer $\varepsilon$, and $\kappa$ is the residue field, and $\mathcal{K}$ is the quotient field.}
\item $\k$ is an algebraically closed field. 
\item All modules are right.  
For an algebra $\Lambda$, we denote by $\sf{mod}$-$\Lambda$ the category of finitely generated $\Lambda$-modules.
\item Tensor products are taken over $\mathcal{O}$. 
\item For an additive category $\mathscr{C}$, let $\underline{\mathscr{C}}$ be the projectively stable category of $\mathscr{C}$. 
\item The symbol $\delta_{i,j}$ means the Kronecker delta.
\item The identity matrix of size $n$ and the zero matrix of size $n$ are denoted by $\mathbf{1}_n$ and $\mathbf{0}_n$, respectively.
\end{enumerate}

\subsection{Orders and lattices}
First, we recall some terminologies on orders and lattices; for example, see \cite[Introduction, Section 1]{I}. 
An $\mathcal{O}$-algebra is called an \textit{$\mathcal{O}$-order} if it is free of finite rank as an $\mathcal{O}$-module. 
An $\mathcal{O}$-order $A$ is called \textit{Gorenstain} if {$\mathsf{Hom}_\mathcal{O}(A,\mathcal{O})$} is a projective $A$-module, and $A$ is said to be \textit{symmetric} if {$\mathsf{Hom}_\mathcal{O}(A,\mathcal{O})\simeq A$} as an $(A,A)$-bimodule. 
For an $\mathcal{O}$-order $A$, an $A$-module $M$ is called an \textit{$A$-lattice} if $M$ is free of finite rank as an $\mathcal{O}$-module. 
We write $\sf{latt}$-$A$ for the full subcategory of $\mathsf{mod}$-$A$ consisting of $A$-lattices. 
{For every indecomposable $A$-lattice $X$, the endomorphism algebra $\mathsf{End}_A(X)$ is a local $\mathcal{O}$-algebra \cite[(6.10) and (30.5)]{CR}.
Moreover, $\sf{latt}$-$A$ has enough projective \cite[(6.23)]{CR}.
}
Then, we define $\sf{latt}^{(\natural)}$-$A$ to be the full subcategory of $\sf{latt}$-$A$ consisting of any $A$-lattice $M$ such that $M\otimes\mathcal{K}$ is projective as an $A\otimes\mathcal{K}$-module. 
By definition, the category $\sf{latt}^{(\natural)}$-$A$ is enough projective and closed under direct summands. In addition, if $A$ is symmetric, the category $\sf{latt}^{(\natural)}$-$A$ is closed under extension.
In this paper,  we write $\overline{A}$ for the finite-dimensional algebra $A\otimes\kappa$. 
The syzygy functors on $\underline{\sf{latt}}$-$A$ and $\underline{\sf{mod}}$-$\overline{A}$ are denoted by $\Omega$ and $\widetilde{\Omega}$, respectively.  

Clearly, the following lemma holds.
\begin{lemma} Let $R$ be an $\mathcal{O}$-order and $M$ an indecomposable $R$-lattice. If $f\in\mathrm{End}_{R}(M)$ is surjective, then $f$ is an isomorphism. Moreover, the set of non-surjective endomorphisms of $M$ coincides with the radical of the endomorphism ring of $M$.
 \end{lemma}

\subsection{Valued stable translation quivers}

In this subsection, we recall notations on stable translation quivers.
A quiver $Q=(Q_0, Q_1,s,t)$ is a quadruple consisting of two sets $Q_0$ and $Q_1$, and two maps $s,t:Q_1\to Q_0$. 
Each element of $Q_0$ and $Q_1$ is called a vertex and an arrow, respectively. 
For an arrow $\alpha\in Q_1$, we call $s(\alpha)$ and $t(\alpha)$ the source and the target of $\alpha$, respectively. 
We write $\overline{Q}$ for the underlying graph of $Q$. 
Given two quivers $Q$ and $\Delta$, a quiver homomorphism $f:Q\to \Delta$ is a pair of maps $f_0: Q_0\to \Delta _0$ and $f_1:Q_1\to \Delta_1$ such that $(s\times t)\circ f_1=(f_0\times f_0)\circ (s\times t)$. 
From now on, we assume that quivers have no multiple arrows; that is, the map $(s\times t)$ is injective.

Let $(Q,v)$ be a pair of a quiver $Q$ and a map $v:Q_{1}\rightarrow \Z_{\geq 0}\times \Z_{\geq 0}$. For an arrow $x\to y$ of $Q$, we write $v(x\to y)=(d_{xy}, d_{xy}')$, and we understand that there is no arrow from $x$ to $y$ if and only if $d_{xy}=d_{xy}'=0$. 
Then, $(Q,v)$ is called a \textit{valued quiver} if $d_{x,y}=0$ if and only if $d_{x,y}'=0$, and the values of the map $v$ are called \textit{valuations}. If $v(x\to y)=(1,1)$ for all arrows $x\to y$ of $Q$, then $v$ is said to be \textit{trivial}. 
We usually omit {writing} trivial valuations. 
For each vertex $x\in Q_0$, we set
\[ x^{+}=\{y\in Q_0\ |\ x\to y \in Q_1\},\quad x^{-}=\{y\in Q_0\ |\  y\to x\in Q_1\}. \]
Note that a quiver is determined by the sets $x^{+}$. 
A quiver $Q$ is \textit{locally finite} if $x^{+}\cup x^{-}$ is a finite set for any $x\in Q_0$. 
A \textit{translation quiver} is a triple $(Q,Q_0',\tau)$ of a locally finite quiver $Q$, a subset $Q_0'\subset Q_0$ and an injective map $\tau:Q_0'\to Q_0$ satisfying $x^{-}=(\tau x)^{+}$. If $Q_0'=Q_0$ and $\tau$ is bijective, the translation quiver is said to be \textit{stable}. 
Then, we simply write $(Q,\tau)$ for the stable translation quiver.
Let $\mathcal{C}$ be a full subquiver of a stable translation quiver $(Q,\tau)$. Then, $\mathcal{C}$ is a (connected) \textit{component} if the following three conditions are satisfied.
\begin{enumerate}[(i)]
\item $\mathcal{C}$ is stable under the quiver automorphism $\tau$.
\item $\mathcal{C}$ is a disjoint union of connected components of the underlying undirected graph.
\item  There is no proper subquiver of $\mathcal{C}$ that satisfies (i) and (ii).
\end{enumerate}

A quiver homomorphism $f$ from a translation quiver $(Q,Q_0',\tau)$ to a translation quiver $(\Delta,\Delta_0',\tau')$ is a \textit{translation quiver homomorphism} if $f_0\circ \tau=\tau'\circ f_0$ is satisfied on $Q_0'$. It is easily seen that $\tau$ induces a translation quiver automorphism when $(Q,Q_0',\tau)$ is stable, and we use the same letter $\tau$. In this paper, we denote by $\mathsf{Aut}_{\tau}(Q)$ the set of all translation quiver automorphisms of $(Q,\tau)$.
Let $Q$ and $\Delta$ be two stable translation quivers. A surjective translation quiver homomorphism $f:Q\to \Delta$ is a \textit{covering} if $f|_{x^{+}}$ gives a bijection between $x^{+}$ and $(f(x))^+$.

For a stable translation quiver $(Q,\tau)$ and a subgroup $G\subset \sf{Aut}_\tau(Q)$, we define the translation quiver homomorphism $\pi_G:Q\to Q/G$ by $\pi_G(x)=Gx$ for $x\in Q_0$.
A subgroup $G\subset\sf{Aut}_\tau(Q)$ is \textit{admissible} if each $G$-orbit intersects $x^+\cup\{x\}$ in at most one vertex and $x^-\cup\{x\}$ in at most one vertex, for any $x\in Q_0$. Then, the map $\pi_G$ is covering.

\begin{definition}
A \textit{valued stable translation quiver} is a triple $(Q, v, \tau)$ such that
\begin{enumerate}[(i)]
\item $(Q,v)$ is a valued quiver, 
\item $(Q,\tau)$ is a stable translation quiver,
\item  $v(\tau y\to x)=(d_{xy}',d_{xy})$ for each arrow $x\to y$.
\end{enumerate}
\end{definition}

Given a valued quiver $(Q, v)$, one can construct the valued stable translation quiver $(\Z Q,\tilde{v},\tau)$ as follows.
\begin{itemize}
\item  $(\Z Q)_0: = \Z\times Q_0$.
\item $(n,x)^+:=\{(n, y)\ |\ y\in x^{+}\}\cup\{(n-1,z)\ |\ z\in x^{-}\}$.
\item $\tilde{v}((n,x)\to (n,y))=(d_{xy},d_{xy}')$, $\tilde{v}((n-1,y)\to (n,x))=(d_{xy}',d_{xy}).$
\item $\tau((n,x))=(n-1,x)$.
\end{itemize}
We write it simply $\Z Q$. Note that $\Z Q$ has no loops whenever $Q$ has no loops. The following theorem is well-known, and it is effective to describe the structure of stable translation quivers; for example, {see \cite[Section 1, p.p. 206]{Ri} or \cite[Theorem 4.15.6]{B}.}

\begin{theorem}[Riedtmann's structure theorm]\label{Ried}
Let $(Q,\tau)$ be a stable translation quiver without loops and $\mathcal{C}$ a connected component of $(Q,\tau)$. Then, there exist a directed tree $T$ and an admissible group $G\subseteq \mathsf{Aut}(\Z T)$ such that $\mathcal{C}\simeq \Z T/G$ as stable translation quivers. Moreover, $\overline{T}$ is uniquely determined by $\mathcal{C}$, and the admissible group is unique up to conjugation.
\end{theorem}
In Theorem \ref{Ried}, the underlying undirected tree $\overline{T}$ is called the \textit{tree class} of $\mathcal{C}$.  

Let $(Q,\tau)$ be a connected stable translation quiver. A vertex $x\in Q_0$ is called \textit{periodic} if $x=\tau ^k x$ for some $k>0$, where $\tau^k$ is the composition of $k$ copies of $\tau$ \cite[p.p. 287]{HPR}.
It is well-known that if there is a periodic vertex in $Q$, then all vertices of $Q$ are periodic. 
Indeed, if $x$ is a periodic vertex in $Q$, then there is a positive integer {$n_x$} such that $\tau^{n_x}x=x$. Since $(Q,\tau)$ is a stable translation quiver, $\tau^{n_x}$ induces a permutation on the finite set $x^{+}$, and so some power of $\tau^{n_x}$ stabilizes $x^{+}$ elementwise. Hence, all vertices in $x^{+}$ are periodic. It follows that all vertices are periodic. In this case, $(Q,v,\tau)$ is called \textit{periodic}. 

A connected valued quiver $(Q,v)$ which has no loops gives rise to a Cartan matrix on $Q_0$: 
\[ C(x,y)=\left\{\begin{array}{ll}
2 & \text{if $x=y$},\\
-d_{x,y}' & \text{if there is an arrow $x\to y$},\\
-d_{y,x} & \text{if there is an arrow $y\to x$},\\
0 & \text{otherwise.}\end{array}\right.\]

\begin{definition}\label{def of subadditive} Let $C$ be a Cartan matrix on $I$. A \textit{subadditive function} for $C$ is a function $\ell: I\to \mathbb{Q}_{>0}$ such that it satisfies
\[ \sum_{y\in I}C(x,y)\ell(y)\geq 0 \]
for all  $x\in I$. A subadditive function $\ell$ is called \textit{additive} if the equality holds for all $x\in I$. We say that a connected valued quiver $Q$ admits a subadditive function when there exists a subadditive function for a Cartan matrix on $Q_0$.
\end{definition}

\begin{remark} Let $(Q,v,\tau)$ be a connected valued translation quiver without loops, and let $\overline{T}$ be the tree class of $Q$. If a function $\ell: Q_0\to \mathbb{Q}_{>0}$ satisfies $\ell(\tau x)=\ell(x)$ and
\[2\ell (x)\ge \sum_{y\to x \text{ in } T} d_{yx}\ell (y)+\sum_{x\to y \text{ in }T} d'_{xy}\ell (y), \]
then the restriction $\ell|_{T}$ is a subadditive function for a Cartan matrix on $\overline{T}_0$. 
\end{remark}

The following result is well-known.
\begin{theorem}[{\cite[Theorem 4.5.8]{B}}]\label{tree class}
Let $(\Delta,v)$ be a connected valued quiver without loops. If $\Delta$ admits a subadditive function $\ell$,
then the following statements hold.
\begin{enumerate}[(1)]
\item The underlying undirected graph $\overline{\Delta}$ is either a finite or infinite Dynkin diagram or a Euclidean diagram.
\item If $\ell$ is not additive, then $\overline{\Delta}$ is either a finite Dynkin diagram or $A_{\infty}$.
\item If $\ell$ is additive, then $\overline{\Delta}$ is either an infinite Dynkin diagram or a Euclidean diagram.
\item If $\ell$ is unbounded, then $\overline{\Delta}$ is $A_{\infty}$.
\end{enumerate}
\end{theorem}

\subsection{The stable AR quiver for the category of lattices over an $\mathcal{O}$-order}
In this subsection, {we recall how to compute almost split sequences and stable Auslander--Reiten components,} see \cite{AKM} for details. 
{Let $A$ be a Gorenstain $\mathcal{O}$-order such that $A\otimes\mathcal{K}$ is self-injective.
Then, the category $\sf{latt}^{(\natural)}$-$A$ admits almost split sequences \cite[Theorems 2.1, and 2.2]{AR}.
We denote by $\tau$ the Auslander--Reiten translation on $\sf{latt}^{(\natural)}$-$A$.}

It is natural to ask how we compute almost split sequences. 
\begin{proposition}[{\cite[Proposition 1.15]{AKM}}]\label{AKM}
Let $A$ be a Gorenstein $\mathcal{O}$-order, $M$ an indecomposable $A$-lattice in $\sf{latt}^{(\natural)}$-$A$, and let $p:P\to M$ be the projective cover of $M$. Given an endomorphism $\varphi : M \to M$, we obtain the pullback diagram along $\nu(p)$ and $\varphi$:
$$\begin{xy}
(0,15)*[o]+{0}="01",(20,15)*[o]+{\mathrm{Ker}(\nu(p))}="L",(40,15)*[o]+{E}="E", (60,15)*[o]+{M}="M",(80,15)*[o]+{0}="02",
(0,0)*[o]+{0}="03",(20,0)*[o]+{\mathrm{Ker}(\nu(p))}="L2",(40,0)*[o]+{\nu(P)}="nP", (60,0)*[o]+{\nu(M)}="nM",(80,0)*[o]+{0}="04",
\ar "01";"L"
\ar "L";"E"
\ar "E";"M"
\ar "M";"02"
\ar "03";"L2"
\ar "L2";"nP"
\ar "nP";"nM"_{\nu(p)}
\ar "nM";"04"
\ar @{-}@<0.5mm>"L";"L2"
\ar @{-}@<-0.5mm>"L";"L2"
\ar "E";"nP"
\ar "M";"nM"^{\varphi}
\end{xy}$$
Here, $\nu=\sf{D}(\mathrm{Hom}_{A}(-,A))$ is the Nalayama functor.
Then, the following statements are equivalent.
\begin{enumerate}[(1)]
\item The upper short exact sequence is the almost split sequence ending at $M$.
\item The following three conditions hold.
\begin{enumerate}[(i)]
\item The morphism $\varphi$ does not factor through $\nu(p)$.
\item $\mathrm{Ker}(\nu (p))$ is an indecomposable $A$-lattice.
\item For all $f\in \mathsf{rad}\mathsf{End}_A(M)$, the morphism $\varphi \circ f$ factors through $\nu(p)$.
\end{enumerate} 
\end{enumerate}
Moreover, any almost split sequence ending at $M$ is given in this way.
\end{proposition}

\begin{remark}[{\cite[Chapter I, Section 2]{H}}]
As $A$ is a Gorenstein $\mathcal{O}$-order, the Nakayama functor $\nu:\sf{latt}$-$A\to\sf{latt}$-$A$ is an autofunctor, and $\underline{\sf{latt}}$-$A$ is a Frobenius category. Hence, $\underline{\sf{latt}}\text{-}A$ is a triangulated category with the shift functor $\Omega^{-1}$.  Then, we have a triangulated equivalence $\nu:\underline{\sf{latt}}\text{-}A\xrightarrow{\sim}\underline{\sf{latt}}\text{-}A$, and the Auslander--Reiten translation $\tau$ is represented by $\Omega\nu$.
\end{remark}

\begin{definition} Let $A$ be an $\mathcal{O}$-order and $M$ be an indecomposable $\overline{A}$-module. We call each direct summand of $\Omega(M)$ \textit{a Heller lattice} of $M$. Note that $\Omega(M)$ may not be an indecomposable $A$-lattice.
\end{definition}

{It is easy to check that any Heller lattices belong to $\sf{latt}^{(\natural)}$-$A$ (see, \cite[Remark 1.12]{AKM}).}

The following proposition is used in this paper everywhere.
\begin{proposition}[{\cite[Proposition 4.5]{K}}]\label{split}
Let $A$ be an $\O$-order and $L$ an indecomposable $A$-lattice, and let
\[ 0\rightarrow \tau L \rightarrow  E \rightarrow L\rightarrow 0 \]
be the almost split sequence {ending} at $L$. Assume that $L$ is not a direct summand of any Heller lattices. Then, the induced exact sequence  
\[ 0\rightarrow \tau L\otimes\kappa \rightarrow E \otimes\kappa \rightarrow L\otimes\kappa\rightarrow 0 \]  
splits.
\end{proposition}

{Next lemma gives a connection between the first syzygies on $A$ and those on $\overline{A}$.}

\begin{lemma}\label{proj cover}
Suppose that $A$ is a symmetric $\mathcal{O}$-order. Then, for any non-projective $A$-lattice $M$, there is an isomorphism $\tau(M)\otimes\kappa\simeq \widetilde{\Omega}(M\otimes\kappa)$.
\end{lemma}
\begin{proof}
Let $M$ be an $A$-lattice and $\pi: P\to M$ the projective cover.
Let $Q\otimes\kappa\to M\otimes\kappa$ be the projective cover. Then $\mathrm{rank}\ Q\leq \mathrm{rank}\ P$. On the other hand, it lifts to $Q\to M$, and it is an epimorphism by Nakayama's lemma. Thus, we have $\mathrm{rank}\ Q=\mathrm{rank}\ P$ and $P\otimes\kappa$ is the projective cover of {$M\otimes\kappa$.}  Therefore, we have $\tau(M)\otimes\kappa\simeq \widetilde{\Omega}(M\otimes\kappa)$ as objects in the stable module category $\underline{\sf{mod}}$-$\overline{A}$. Since the functor $-\otimes\kappa$ is exact on $\sf{latt}$-$A$, the assertion follows.
\end{proof}

\begin{definition}\label{ARquiver} Let $A$ be a symmetric $\mathcal{O}$-order.
\begin{enumerate}[(1)]
\item The \textit{stable Auslander--Reiten quiver} for $\sf{latt}^{(\natural)}$-$A$ is the valued quiver defined as follows:
\begin{itemize}
\item The set of vertices is a complete set of isoclasses of non-projective indecomposable $A$-lattices in $\sf{latt}^{(\natural)}$-$A$.
\item We draw a valued arrow $M\xrightarrow{(a,b)}N$ whenever there exist irreducible morphisms $M\to N$, where the valuation $(a,b)$ means:
\begin{enumerate}[(i)]
\item $a$ is the multiplicity of $M$ in the middle term of the almost split sequence ending at $N$. 
\item $b$ is the multiplicity of $N$ in the middle term of the almost split sequence starting at $M$. 
\end{enumerate}
\end{itemize}
The stable Auslander--Reiten quiver for $\sf{latt}^{(\natural)}$-$A$ is denoted by $\Gamma_s(A)$ in this paper.
\item A component of $\Gamma_s(A)$ containing an indecomposable Heller lattice $Z$ is said to be a \textit{Heller component} of $A$, and denoted by $\mathcal{CH}(Z)$. 
\end{enumerate}
\end{definition}
 
By the definition, we note that a component $\mathcal{C}$ of $\Gamma_s(A)$ does not have multiple arrows, and $\tau M$ and $\tau^{-1}M$ exist for each vertex $M$ of $\mathcal{C}$ (\cite[Theorems 2.1, and 2.2]{AR}).
Thus, $(\mathcal{C},\tau)$ is a valued stable translation quiver. Note that there are possibilities that $\mathcal{C}$ has loops {\cite[Theorems 1 and 2]{W}.}
{However, if a component $\mathcal{C}$ with infinitely many vertices admits loops, then the shape of $\mathcal{C}$ is entirely determined.}

\begin{proposition}[{\cite[Lemma 1.21 and Theorem 1.27]{AKM19}}]\label{periodic_case}
Let $A$ be a symmetric $\O$-order and $\mathcal{C}$ a periodic component of $\Gamma_s(A)$. Assume that $\Gamma_s(A)$ has infinitely many vertices. Then, one of the following statements holds.
\begin{enumerate}[(1)]
\item If $\mathcal{C}$ has loops, then $\mathcal{C}\setminus\{\text{loops}\}=\mathbb{Z}A_\infty/\langle \tau \rangle$. Moreover, the loops only appear on the boundary of $\mathcal{C}$.
\item If $\mathcal{C}$ has no loops, then the tree class of $\mathcal{C}$ is one of infinite Dynkin diagrams.\end{enumerate}
\end{proposition}

\subsection{Indecomposable modules over special biserial algebras}\label{indec SBA}
For a finite-dimensional algebra $\Lambda$, let $Q$ be the Gabriel quiver, $\mathcal{I}$ the admissible ideal such that $\Lambda\simeq \k Q/\mathcal{I}$. Then, $\Lambda$ is called \textit{special biserial} if the following two conditions are satisfied.
\begin{enumerate}[(i)]
\item For each vertex $x$ of $Q$, $\sharp x^{+}\leq 2$ and $\sharp x^{-}\leq 2$.
\item For each arrow $\alpha$ of $Q$, there exist at most one arrow $\beta$ such that $\alpha\beta\notin\mathcal{I}$ and at most one arrow $\gamma$ such that $\gamma\alpha\notin\mathcal{I}$.
\end{enumerate}
Special biserial algebras are of tame representation type, and all finite-dimensional indecomposable modules are classified into ``string modules" and ``band modules"; (for the definitions of string modules and band modules, for example, see {\cite[Section II, p.p. 47--52]{Erd}.}) 

\begin{theorem}[{\cite[(2.3) Proposition]{WW}}]\label{stringband}
Let $\Lambda$ be a special biserial algebra. 
Then, the disjoint union of the set of string modules, the set of band modules, and the set of all projective-injective modules corresponding to the binomial relations forms a complete set of isoclasses of finite-dimensional indecomposable $\Lambda$-modules.
\end{theorem}

\section{The Heller lattices over the symmetric Kronecker algebra}
In this section, we consider the symmetric Kronecker algebra $A:=\O[X, Y]/(X^2, Y^2)$, that is, it is the bound quiver algebra over $\O$ defined by the following quiver and relations:
$$ \begin{xy}
(0,0)*[o]+{1}="1",
\SelectTips{eu}{}
\ar @(ur,dr)"1";"1"^{Y}
\ar @(ul,dl)"1";"1"_{X}
\end{xy};\quad X^2=Y^2=0, \quad XY-YX=0.$$
From this section to the end of this paper, we assume that {the residue field} $\kappa$ is algebraically closed. Then, a $d$-dimensional $\overline{A}$-module $M$ is of the form
$$ M=\begin{xy}
(0,0)*[o]+{\kappa ^d}="1",
\SelectTips{eu}{}
\ar @(ur,dr)"1";"1"^{M_2}
\ar @(ul,dl)"1";"1"_{M_1}
\end{xy},$$
where $M_1$ and $M_2$ are square matrices of size $d$ which commute and have square zero \cite[Chapters II and III]{ASS}. To simplify, we denote by $(d, M_1, M_2)$ the $\overline{A}$-module $M$. Throughout this section, for a positive integer $n$, we denote by $e_1,\ldots e_n$ the standard basis of $\O ^n$ and we adopt $e_1, Xe_1, Ye_1, XYe_1,\ldots ,e_n, Xe_n, Ye_n, XYe_n$ as an $\O$-basis of $A^n$. We call this $\O$-basis of $A^n$ \textit{the standard basis} of $A^n$. 

\subsection{Indecomposable modules over \mbox{\boldmath $\overline{A}$}}\label{indec mod Heller}

In this subsection,  we give a complete list of Heller lattices. By Theorem \ref{stringband}, all finite-dimensional indecomposable $\overline{A}$-modules are classified into string modules, band modules, and projective-injective modules. 

First, the unique indecomposable projective-injective module $\overline{A}$ is given by
\[ \left(4, 
\left(\begin{array}{cccc}
0 & 0& 0 & 0 \\
1 & 0 & 0 & 0 \\
0 & 0 & 0 & 0 \\
0 & 0 & 1 & 0 \end{array}\right), \left(\begin{array}{cccc}
0 & 0& 0 & 0 \\
0 & 0 & 0 & 0 \\
1 & 0 & 0 & 0 \\
0 & 1 & 0 & 0 \end{array}\right) \right). \]
Now, we present a complete list of the other finite-dimensional indecomposable $\overline{A}$-modules, which are denoted by $M(m)$, $M(-m)$, $M(\lambda)_{n}$, where $m\in\Z_{\geq 0}$, $n>0$ and $\lambda$ {in} the projective line $\mathbb{P}^1(\kappa)=\kappa\sqcup\{\infty\}$. 
\begin{enumerate}[(i)]
\item The string module $M(m):=M((\beta_1^{\ast}\beta_2)^m)$ $(m\in\mathbb{Z}_{\geq 0})$ is given by the formula:
\[ M(m) =\left( 2m+1, 
\left( 
\begin{array}{c:c}
\mathbf{0}_m & \mathbf{0}_{m+1} \\
\hdashline %
\mathbf{1}_m & \mathbf{0}_{m+1} \\
0  \cdots  0 & 0 \cdots 0 
\end{array}
\right),  
\left( 
\begin{array}{c:c}
\mathbf{0}_m & \mathbf{0}_{m+1} \\
\hdashline %
0  \cdots  0 & 0 \cdots 0  \\
\mathbf{1}_m & \mathbf{0}_{m+1} \\
\end{array}
\right)\right)
\]
 \item The string module $M(-m):=M((\beta_1\beta_2^\ast)^m)$ $(m\in\mathbb{Z}_{\geq 0})$ is given by the formula:
\[ M(-m) =\left( 2m+1, 
\left( 
\begin{array}{cc:c}
&\hspace{-1cm}\mathbf{0}_{m+1} & \mathbf{0}_{m} \\
\hdashline %
\mathbf{1}_m &\begin{array}{c}
\vspace{-2mm}0 \\
\vdots \\
0 \end{array} & \mathbf{0}_{m} \\
\end{array}
\right),  
\left( 
\begin{array}{cc:c}
&\hspace{-1cm}\mathbf{0}_{m+1} & \mathbf{0}_{m} \\
\hdashline %
\begin{array}{c}
\vspace{-2mm}0 \\
\vdots \\
0 \end{array} &\mathbf{1}_m & \mathbf{0}_{m} \\
\end{array}
\right)  \right)
\]

\item The string module $M(0)_{n}:=M((\beta_1\beta_2^{\ast})^{n-1}\beta_1)$ $(n\in\mathbb{Z}_{>0})$ is given by the formula:
\[  M(0)_{n} = \left( 2n, 
\left( 
\begin{array}{c:c}
\mathbf{0}_n & \mathbf{0}_{n} \\
\hdashline %
\mathbf{1}_n & \mathbf{0}_{n} \\
\end{array}
\right),  
\left( 
\begin{array}{c:c}
\mathbf{0}_n & \mathbf{0}_{n} \\
\hdashline %
J(0,n) & \mathbf{0}_{n} \\
\end{array}
\right)\right)\]

\item The string module $M(\infty)_{n}:=M(\beta_2(\beta_1^{\ast}\beta_2)^{n-1})$ $(n\in\mathbb{Z}_{>0})$ is given by the formula:
\[  M(\infty)_{n} = \left( 2n, 
\left( 
\begin{array}{c:c}
\mathbf{0}_n & \mathbf{0}_{n} \\
\hdashline %
J(0,n) & \mathbf{0}_{n} \\
\end{array}
\right),  
\left( 
\begin{array}{c:c}
\mathbf{0}_n & \mathbf{0}_{n} \\
\hdashline %
\mathbf{1}_n& \mathbf{0}_{n} \\
\end{array}
\right)\right)\]

\item Let $V$ be a finite-dimensional indecomposable left $\kappa[x,x^{-1}]$-module. Assume that $V$ is represented by $x\mapsto J(\lambda, n)$ with respect to a basis of $V$ for some $\lambda\in\kappa^{\times}$ and $n\in\mathbb{Z}_{>0}$. The band module $M(\lambda)_{n}:=N(\beta_2^{\ast}\beta_1, V)$ is given by the formula:
\[ M(\lambda)_{n} = \left( 2n, 
\left( 
\begin{array}{c:c}
\mathbf{0}_n & \mathbf{0}_{n} \\
\hdashline %
\mathbf{1}_n & \mathbf{0}_{n} \\
\end{array}
\right),  
\left( 
\begin{array}{c:c}
\mathbf{0}_n & \mathbf{0}_{n} \\
\hdashline %
J(\lambda,n)& \mathbf{0}_{n} \\
\end{array}
\right)\right)\]
\end{enumerate}

\begin{lemma}\label{indecc}
The set of the $\overline{A}$-modules 
\[ \{ M(m) \mid m\in\mathbb{Z}\}\sqcup\{ M(\lambda)_{n} \mid \lambda\in \mathbb{P}^{1}(\kappa),\ n\in\mathbb{Z}_{\geq 1}\} \sqcup \{\overline{A}\} \]
forms a complete set of isoclasses of finite-dimensional indecomposable modules over $\overline{A}$.
\end{lemma}
\begin{proof}
The assertion follows from {Theorem} \ref{stringband}. See also \cite{M} for a construction.
\end{proof}

\para\label{notation} \textbf{Notation.}
For simplicity, we visualize an $\overline{A}$-module as follows:
\begin{itemize}
\item Vertices represent basis vectors of the underlying $\kappa$-vector spaces.
\item Arrows of the form $\longrightarrow$ represent the action of $X$, and $\dashrightarrow$ represent the action of $Y$. 
\item If there is no arrow (resp. dotted arrow) starting at a vertex, then $X$ (resp. $Y$) annihilates the corresponding basis element.
\end{itemize}
By using this notation, the indecomposable modules listed above are represented as follows:
\begin{multienumerate}
\mitemxx{$\overline{A}= 
\begin{xy}
(0,0)*[o]+{e_1}="1",(15,7)*[o]+{Xe_1}="X",(15,-7)*[o]+{Ye_1}="Y",(30,0)*[o]+{XYe_1}="XY",
\ar @{-->}"1";"Y"
\ar "1";"X"
\ar @{-->} "X";"XY"
\ar "Y";"XY"
\end{xy}. 
$}{$ M(m) =\begin{xy}
(0,7)*[o]+{u_1}="u1",(20,12)*[o]+{v_0}="v0",(20,-7)*[o]+{v_{m-1}}="vp-1",(0,-7)*[o]+{u_{m-1}}="up-1",
(0,2)*[o]+{\vdots}="2",(0,-2)*[o]+{\vdots}="22",
(0,-12)*[o]+{u_m}="up",(20,7)*[o]+{v_1}="v1",(20,-12)*[o]+{v_m}="vp",
(20,2)*[o]+{\vdots}="11",(20,-2)*[o]+{\vdots}="111",
\ar @{-->}"up";"vp"
\ar "up";"vp-1"
\ar @{-->}"u1";"v1"
\ar @{-->}"up-1";"vp-1"
\ar "u1";"v0"
\end{xy}$}
\mitemxx{$M(-m) = \begin{xy}
(0,7)*[o]+{u_1}="u1",
(20,-7)*[o]+{v_{m-1}}="vp-1",(0,-12)*[o]+{u_{m}}="up-1",
(0,-3)*[o]+{\vdots}="2",(0,2)*[o]+{u_2}="u2",(20,2)*[o]+{v_2}="v2",
(0,-17)*[o]+{u_{m+1}}="up",(20,-12)*[o]+{v_{m}}="vp",
(20,7)*[o]+{v_1}="v1",
(20,-2)*[o]+{\vdots}="11",
\ar @{-->}"up";"vp"
\ar @{-->}"up-1";"vp-1"
\ar @{-->}"u2";"v1"
\ar "u1";"v1"
\ar "u2";"v2"
\ar "up-1";"vp"
\end{xy}
$}{$M(0)_{n} =
 \begin{xy}
(0,7)*[o]+{u_1}="u1",
(20,-7)*[o]+{v_{n-1}}="vp-1",(0,-7)*[o]+{u_{n-1}}="up-1",
(0,-2)*[o]+{\vdots}="2",(0,2)*[o]+{u_2}="u2",(20,2)*[o]+{v_2}="v2",
(0,-12)*[o]+{u_{n}}="up",(20,-12)*[o]+{v_{n}}="vp",
(20,7)*[o]+{v_1}="v1",
(20,-2)*[o]+{\vdots}="11",
\ar "up";"vp"
\ar "up-1";"vp-1"
\ar @{-->}"u2";"v1"
\ar "u1";"v1"
\ar "u2";"v2"
\ar @{-->}"up";"vp-1"
\end{xy}
$}
\mitemxx{$M(\infty)_{n}=
\begin{xy}
(0,12)*[o]+{u_1}="u0",
(0,7)*[o]+{u_2}="u1",(20,12)*[o]+{v_1}="v0",(20,-7)*[o]+{v_{n-1}}="vp-1",(0,-7)*[o]+{u_{n-1}}="up-1",
(0,2)*[o]+{\vdots}="2",(0,-2)*[o]+{\vdots}="22",
(0,-12)*[o]+{u_n}="up",(20,7)*[o]+{v_2}="v1",(20,-12)*[o]+{v_n}="vp",
(20,2)*[o]+{\vdots}="11",(20,-2)*[o]+{\vdots}="111",
\ar @{-->}"up";"vp"
\ar @{-->}"u0";"v0"
\ar "up";"vp-1"
\ar @{-->}"u1";"v1"
\ar @{-->}"up-1";"vp-1"
\ar "u1";"v0"
\end{xy} $}{$ M(\lambda)_{n}= \begin{xy}
(0,14)*[o]+{u_1}="u1",
(20,-7)*[o]+{v_{n-1}}="vp-1",(0,-7)*[o]+{u_{n-1}}="up-1",
(0,0)*[o]+{\vdots}="2",(0,7)*[o]+{u_2}="u2",(20,7)*[o]+{v_2}="v2",
(0,-14)*[o]+{u_{n}}="up",(20,-14)*[o]+{v_{n}}="vp",
(20,14)*[o]+{v_1}="v1",
(20,0)*[o]+{\vdots}="11",
\ar "up";"vp"
\ar "up-1";"vp-1"
\ar @{-->}"u2";"v1"
\ar @(ur,ul)@{-->}"u1";"v1"^{\lambda}
\ar "u1";"v1"
\ar @(dr,dl)@{-->}"u2";"v2"^{\lambda}
\ar "u2";"v2"
\ar @{-->}"up";"vp-1"
\ar @(dr,dl)@{-->}"up";"vp"^{\lambda}
\end{xy}
$}
\end{multienumerate}
Here, $\begin{xy} 
(0,-5)*[o]+{u_i}="u",(20,-5)*[o]+{v_{i}}="v1",(20,5)*[o]+{v_{i-1}}="v2"
\ar @{-->}"u";"v2"
\ar @(dr,dl)@{-->}"u";"v1"^{\lambda}\end{xy}
$ in the picture 6 means $Yu_i=\lambda v_i+v_{i-1}$. 

From now on, as a $\kappa$-basis of a non-projective indecomposable module over $\overline{A}$, we adopt the above $\kappa$-basis.

\begin{remark}\label{ARseq} Almost split sequences for $\sf{mod}$-$\overline{A}$ are known to be as follows:
$$
\begin{array}{ll}
0\longrightarrow M(-1)\longrightarrow \overline{A}\oplus M(0)\oplus M(0)\longrightarrow M(1)\longrightarrow 0&\\
0\longrightarrow M(n-1) \longrightarrow M(n)\oplus M(n)\longrightarrow M(n+1)\longrightarrow 0 & \text{if $n\neq 0$}\\
0\longrightarrow M(\lambda)_{1} \longrightarrow   M(\lambda)_{2}\longrightarrow M(\lambda)_{1}\longrightarrow 0 &  \lambda\in \mathbb{P}^1(\kappa) \\
0\longrightarrow M(\lambda)_{n} \longrightarrow  M(\lambda)_{n-1}\oplus  M(\lambda)_{n+1}\longrightarrow M(\lambda)_{n}\longrightarrow 0 & n> 1, \lambda\in \mathbb{P}^1(\kappa) 
\end{array}
$$
\end{remark}


\begin{lemma}\label{proj1} For all $\lambda\in \mathbb{P}^1(\kappa)$ and $n\in\mathbb{Z}_{>0}$, there is an isomorphism
\[ \widetilde{\Omega}(M(\lambda)_n)\simeq M(-\lambda)_{n},\quad \widetilde{\Omega}(M(\infty)_n)\simeq M(\infty)_n. \]
\end{lemma}
\begin{proof}
For $\lambda\in\mathbb{P}^1(\kappa)$ and $n>0$, we define a map $\pi_n^{\lambda}:(\overline{A})^n \to M(\lambda)_n$ by $\pi_n^{\lambda}: e_i \mapsto u_i$.
Then, $\pi_n^{\lambda}$ is the projective cover of $M(\lambda)_n$ as an $\overline{A}$-module. First, we assume that $\lambda\neq\infty$. In this case, the kernel of $\pi_n^{\lambda}$ is given by
\begin{align*}
&\ \kappa (Ye_1-\lambda Xe_1) \oplus \kappa XY e_1 \\
&\oplus \kappa(Ye_2-\lambda Xe_2-Xe_1) \oplus\kappa (-XY e_2) \\
&\oplus \cdots \\
&\oplus \kappa(Ye_n-\lambda Xe_n-Xe_{n-1}) \oplus \kappa (-1)^nXY e_n,
\end{align*}
and it is isomorphic to $M(-\lambda)_n$ in $\underline{\sf{mod}}\text{-}\overline{A}$. 
Next, we consider $\lambda=\infty$ case. A $\kappa$-basis of the kernel of $\pi_n^{\infty}$ is given by
\begin{align*}
&\ \kappa Xe_1 \oplus \kappa XY e_1 \\
&\oplus \kappa(Xe_2-Ye_1) \oplus\kappa (-XY e_2) \\
&\oplus \cdots \\
&\oplus \kappa(Xe_n-Ye_{n-1}) \oplus \kappa (-1)^nXY e_n,
\end{align*}
and it is isomorphic to $M(\infty)_n$ in $\underline{\sf{mod}}\text{-}\overline{A}$.
In both cases, the isomorphisms are lifted in $\sf{mod}$-$\overline{A}$ since the kernels have no $\overline{A}$ as a direct summand.
\end{proof}

\subsection{Properties of Heller lattices}\label{Heller lattices}
Let $M$ be a non-projective indecomposable $\overline{A}$-module listed in Lemma \ref{indecc}. 
For each $m$ and $\lambda$, the projective cover of $M$ as an $A$-module $\pi _M$ is given by 
\[ \pi_M:\left\{ \begin{array}{lll}
A^m\longrightarrow M, & e_i\longmapsto u_i & \text{if $M\simeq M(m)$,\quad $m>0$,}\\
A^{m+1}\longrightarrow M, & e_i\longmapsto u_i & \text{if $M\simeq M(-m)$,\quad  $m>0$,}\\
A\longrightarrow M, & e_1\longmapsto u_1& \text{if $M\simeq M(0)$,}\\
A^n\longrightarrow M, & e_i\longmapsto u_i & \text{if $M\simeq M(\lambda)_n$,\quad  $n>0$, $\lambda\in\mathbb{P}^1(\kappa)$.}\\
\end{array}\right. \]

The author studied the Heller lattices of $M(m)$ for $m\in\mathbb{Z}$ and determined the shape of the unique non-periodic Heller component containing them. 

\begin{theorem}[{\cite[{Propositions 2.4, 2.7, and Theorem}]{M}}]\label{nonperiodic Heller}
For each $m\in\mathbb{Z}$, let $Z_m$ be the kernel of $\pi_{M(m)}$. Then the following statements hold.
\begin{enumerate}[(1)]
\item There is an isomorphism $Z_m\otimes \kappa\simeq M(m-1)\oplus M(m)$.
\item The Heller lattice $Z_m$ is indecomposable.
\item There is an isomorphism $\tau Z_m\simeq Z_{m-1}$.
\item The Heller component containing $Z_m$ is isomorphic to $\mathbb{Z}A_{\infty}$. 
\item The Heller lattice $Z_m$ appears on the boundary of the component.
\end{enumerate}
\end{theorem}

In this paper, we focus on the remaining Heller lattices.
For $n\in\mathbb{Z}_{>0}$ and  $\lambda\in\mathbb{P}^1(\kappa)$, we define the Heller $A$-lattice $Z_{n}^{\lambda}$ to be the kernel of $\pi_{M(\lambda)_n}$. 

\para\label{notation2} \textbf{Notation.}
We use the following notations: 

\noindent$\bullet$ For the Heller lattice $Z_n^{\lambda}$ $({n\geq 1},\ \lambda\neq\infty)$, we define
\[ \begin{array}{l} \left(\begin{array}{cccc}
\sf{a}_{1,1}^{\lambda} & \sf{a}_{1,2}^{\lambda}& \sf{a}_{1,3}^{\lambda}& \sf{a}_{1,4}^{\lambda} \\
\sf{a}_{2,1}^{\lambda}& \sf{a}_{2,2}^{\lambda} & \sf{a}_{2,3}^{\lambda} & \sf{a}_{2,4}^{\lambda} \\
\vdots & \vdots & \vdots & \vdots \\
\sf{a}_{n-1,1}^{\lambda}& \sf{a}_{n-1,2}^{\lambda} & \sf{a}_{n-1,3}^{\lambda} & \sf{a}_{n-1,4}^{\lambda} \\
\sf{a}_{n,1}^{\lambda} & \sf{a}_{n,2}^{\lambda} & \sf{a}_{n,3}^{\lambda} & \sf{a}_{n,4}^{\lambda}
\end{array}\right)  \\
 = \left(\begin{array}{cccc}
\varepsilon e_1 & \varepsilon Xe_1 & (Ye_1-\lambda Xe_1) & XY e_1 \\
\varepsilon e_2 & \varepsilon Xe_2 & (Ye_2-\lambda Xe_2-Xe_1) & XY e_2 \\
\vdots & \vdots & \vdots & \vdots \\
\varepsilon e_{n-1} & \varepsilon Xe_{n-1}& (Ye_{n-1}-\lambda Xe_{n-1}-Xe_{n-2}) & XY e_{n-1} \\
\varepsilon e_n & \varepsilon Xe_n  & (Ye_n-\lambda Xe_n-Xe_{n-1}) & XY e_n
\end{array}\right) \end{array} \]
{Here, we understand that $\sf{a}_{0,j}^{\lambda}=0$ for $j=1,2,3,4$.}
Then, $X$ and $Y$ act on $Z_{n}^{\lambda}$ as follows.
\[ X\sf{a}_{i,j}^{\lambda}=\left\{\begin{array}{ll}
\sf{a}_{i,j+1}^{\lambda} & \text{if $j=1,3,$} \\
0 & \text{otherwise,} \end{array}\right.
\quad
 Y\sf{a}_{i,j}^{\lambda}=\left\{\begin{array}{ll}
\varepsilon \sf{a}_{i,3}^{\lambda}+\lambda\sf{a}_{i,2}^{\lambda}+\sf{a}_{i-1,2}^{\lambda}& \text{if $j=1$,} \\
\varepsilon\sf{a}_{i,4}^{\lambda} & \text{if $j=2,$} \\
- \lambda\sf{a}_{i,4}^{\lambda}-\sf{a}_{i-1,4}^{\lambda} & \text{if $j=3,$} \\
0 & \text{otherwise.} 
\end{array}\right.\]

\noindent$\bullet$ For the Heller lattice $Z_n^{\infty}$, we define 
\[ \left(\begin{array}{cccc}
\sf{b}_{1,1} & \sf{b}_{1,2} & \sf{b}_{1,3} & \sf{b}_{1,4} \\
\sf{b}_{2,1} & \sf{b}_{2,2} & \sf{b}_{2,3} & \sf{b}_{2,4} \\
\vdots & \vdots & \vdots & \vdots \\
\sf{b}_{n-1,1} & \sf{b}_{n-1,2} & \sf{b}_{n-1,3} & \sf{b}_{n-1,4} \\
\sf{b}_{n,1} & \sf{b}_{n,2} & \sf{b}_{n,3} & \sf{b}_{n,4}
\end{array}\right)=
 \left(\begin{array}{cccc}
\varepsilon e_1 & Xe_1 & (Ye_1-Xe_2) & XY e_1 \\
\varepsilon e_2 &\varepsilon Xe_2 & (Ye_2-Xe_3) & XY e_2 \\
\vdots & \vdots & \vdots & \vdots \\
\varepsilon e_{n-1} &\varepsilon Xe_{n-1} & (Ye_{n-1}-Xe_n) & XY e_{n-1} \\
\varepsilon e_n &\varepsilon Xe_n & \varepsilon Ye_n & XY e_n 
\end{array}\right) \]
when $n>1$, and if $n=1$, we define
\[ (\sf{b}_{1,1},\sf{b}_{1,2},\sf{b}_{1,3},\sf{b}_{1,4})=(\varepsilon e_1, Xe_1, \varepsilon Ye_1, XYe_1). \]
Then, $X$ and $Y$ act on $Z_{n}^{\infty}$ as follows. If $n>1$, then
\[ X\sf{b}_{i,j}=\left\{\begin{array}{ll}
\varepsilon \sf{b}_{1,2} & \text{if $i=j=1,$} \\
\sf{b}_{i,2} & \text{if $i\neq 1,\ j=1,$} \\
\sf{b}_{i,4} & \text{if $i\neq n,\ j=3,$} \\
\varepsilon \sf{b}_{n,4} & \text{if $i=n,\ j=3,$} \\
0 & \text{otherwise,} \end{array}\right.
\quad
 Y\sf{b}_{i,j}=\left\{\begin{array}{ll}
\varepsilon \sf{b}_{i,3}+\sf{b}_{i+1,2} & \text{if $i\neq n,\ j=1,$} \\
\sf{b}_{n,3} & \text{if $i= n,\ j=1,$} \\
\sf{b}_{1,4} & \text{if $i=1,\ j=2,$} \\
\varepsilon \sf{b}_{i,4} & \text{if $i\neq 1,\ j=2,$} \\
- \sf{b}_{i+1,4} & \text{if $i\neq n,\ j=3,$} \\
0 & \text{otherwise.} 
\end{array}\right.\]
If $n=1$, then
\[ X\sf{b}_{1,j}=\left\{\begin{array}{ll}
\varepsilon \sf{b}_{1,j+1}& \text{if $j=1,3$,}\\
0 & \text{otherwise,}
\end{array}\right.\quad
Y\sf{b}_{1,j}=\left\{\begin{array}{ll}
\sf{b}_{1,j+2} & \text{if $j=1,2$,} \\
0 & \text{otherwise.} \end{array}\right. \]

It is straightforward to prove the following lemma.
\begin{lemma}
We use the lexicographical order on $\{(i,j)\mid i=1,\ldots,n,\ j=1,2,3,4\}$.
Then, the sets $\{ \sf{a}^{\lambda}_{i,j} \mid i=1,\ldots,n,\ j=1,2,3,4\}$ and $\{ \sf{b}_{i,j}\ |\ i=1,\ldots,n,\ j=1,2,3,4\}$ form an (ordered) $\mathcal{O}$-basis of $Z_n^{\lambda}$ and $Z_n^{\infty}$, respectively.
\end{lemma}
{In this paper, when we consider the matrix representation of Heller lattices $Z_n^{\lambda}$ $(n\geq 1, \lambda\in\mathbb{P}^1(\kappa))$, we always use these $\mathcal{O}$-bases.}

\begin{proposition} \label{indec of Heller}
\begin{enumerate}[(1)]
\item For each $\lambda\in\kappa$ and $n>0$, the Heller lattice $Z_n^{\lambda}$ is indecomposable.
\item For each $n>0$, the Heller lattice $Z_n^{\infty}$ is indecomposable.
\item For each $\lambda\in\kappa$ and $n>0$, there is an isomorphism $Z^{\lambda}_n\otimes\kappa\simeq M(\lambda)_n\oplus M(-\lambda)_n$ as $\overline{A}$-modules.
\item For each $n>0$, there is an isomorphism $Z^{\infty}_n\otimes\kappa\simeq M(\infty)_n^{\oplus 2}$ as $\overline{A}$-modules.
\end{enumerate}
\end{proposition}

\para \textbf{Proof of (1) in Proposition \ref{indec of Heller}}
Let $\widetilde{X}$, $\widetilde{Y}$ and $\widetilde{\widetilde{Y}}$ be square matrices of size $4$ defined by
\[ \widetilde{X}:=\left( \begin{array}{cccc}
0 & 0 & 0 & 0 \\
1 & 0 & 0 & 0 \\
0 & 0 & 0 & 0 \\
0 & 0 & 1 & 0 \end{array}\right)\quad \widetilde{Y}:=\left( \begin{array}{cccc}
0 & 0 & 0 & 0 \\
\lambda & 0 & 0 & 0 \\
\varepsilon & 0 & 0 & 0 \\
0 & \varepsilon & -\lambda & 0 \end{array}\right)\quad \widetilde{\widetilde{Y}}:=\left( \begin{array}{cccc}
0 & 0 & 0 & 0 \\
1 & 0 & 0 & 0 \\
0 & 0 & 0 & 0 \\
0 & 0 & -1 & 0 \end{array}\right)\]
Then, {the representing matrices of the actions of $X$ and $Y$ on $Z_n^{\lambda}$ are of the form:}
\[ X=\left(
\begin{array}{ccccccc}
\widetilde{X}&&&&&&\\
&\widetilde{X}&&&&\hsymb{0}&\\
&&&&\ddots&&\\
&&&&&\widetilde{X}&\\
&\hsymb{0}&&&&&\widetilde{X}
\end{array}
\right)\quad Y=\left(
\begin{array}{ccccccc}
\widetilde{Y}&\widetilde{\widetilde{Y}}&&&&&\\
&\widetilde{Y}&\widetilde{\widetilde{Y}}&&&\hsymb{0}&\\
&&&&\ddots&&\\
&&&&&\widetilde{Y}&\widetilde{\widetilde{Y}}\\
&\hsymb{0}&&&&&\widetilde{Y}
\end{array}
\right) \in \mathrm{Mat}(4n,4n, \mathcal{O}) \]

Obviously, the Heller lattice $Z_1^{\lambda}$ is indecomposable since $Z_1^{\lambda}\otimes\mathcal{K}\simeq A\otimes \mathcal{K}$. We prove that idempotents of $\mathsf{End}_A(Z_n^{\lambda})$ are only $\mathbf{1}_{4n}$ and $\mathbf{0}_{4n}$. Let $M=(m_{i,j})$ be an idempotent of $\mathsf{End}_A(Z_n^{\lambda})$. We partition $M$ into $n$ blocks of size $4\times 4$, and denote by $M_{i,j}\in\mathrm{Mat}(4,4,\mathcal{O})$ the $(i,j)$-block of $M$ and by $\alpha_{i,j}$ the $(4i-2,4j-1)$-entry of $M$. The equalities $MX=XM$ and $MY=YM$ yield that the block $M_{ij}$ is of the form 
 \[ M_{i,j}=\left( \begin{array}{cccc}
d_{i,j} & 0 & 0 & 0 \\
m_{4i-2,4j-3} & d_{i,j}  & c_{i,j} & 0 \\
m_{4i-1,4j-3}  & 0 & d_{i,j}  & 0 \\
m_{4i,4j-3}  & m_{4i-1,4j-3}  & m_{4i,4j-1}  & d_{i,j}  \end{array}\right), \]
where 
\begin{equation}\label{dij} d_{i,j}=\left\{\begin{array}{ll}
m_{1,1} & \text{if $i=j=1$,}\\
m_{1,1}+\varepsilon\sum_{k=1}^{j-1}\alpha _{k, k+1} &\text{if $i=j>1$},\\
\varepsilon\sum_{k=1}^{j}\alpha _{i-j-1+k, k} & \text{if $n\geq i>j\geq 1$},\\
m_{1,4j-3} & \text{if $n\geq j>i=1$}, \\
m_{1,4(j-i)+1}+\varepsilon\sum_{k=1}^{i-1}\alpha_{k,j-i+1+k} & \text{if $n\geq j>i>1$},\end{array}\right. \end{equation} 
\[ c_{i,j}=\left\{\begin{array}{ll}
0 & \text{if $i=n, j=1$,}\\
\alpha_{i,j} &\text{if $i\neq n$},\\
-\sum_{k=1}^{j-1}\alpha _{i-j+k, k} & \text{if $n= i\geq j> 1$}.
\end{array}\right. \]
Here, we have to choose each element $m_{k,l}$ in $M_{i,j}$ in such a way that the equation $MY=YM$ holds. By comparing the $(1,1)$-entries of $M$ and $M^2$, we have the equation
\[ m_{1,1}=m_{1,1}^2+\varepsilon \sum_{k=1}^{n-1}m_{1,4k+1}m_{4k-2,3}. \]
We write $\overline{x}$ for the coset in the residue field $\kappa=\mathcal{O}/\varepsilon\mathcal{O}$ represented by $x\in\mathcal{O}$. The above equation implies that $\overline{m_{1,1}}$ is either $\overline{0}$ or $\overline{1}$.  

Assume that $\overline{m_{1,1}}=\overline{0}$. Then, the element $d_{i,i}$ belongs to $\varepsilon\mathcal{O}$ for all $i$ by (\ref{dij}). By comparing the $(1,4k+1)$-entries of $M$ and $M^2$, we have 
\begin{equation}\label{1,4k+1} m_{1,4k+1}=m_{1,1}m_{1,4k+1}+\sum_{l=1}^{k}m_{1,4l+1}d_{l+1,l+1}+\varepsilon\sum_{l=k+1}^{n-1}m_{1,4l+1}P(l) \end{equation}
for some $P(l)\in \mathcal{O}$, and hence $m_{1,4k+1}\in\varepsilon\mathcal{O}$ for all $k$. From (\ref{1,4k+1}), $m_{1,4k+1}$ belongs to $\varepsilon ^t\mathcal{O}$ for all $t>0$. It implies that $m_{1,4k+1}=0$ for all $k$. Therefore, the first row of $M$ is zero.  By comparing the $(5,5)$-entries of $M$ and $M^2$, the following equation holds:
\[ \varepsilon m_{2,7} = \left\{\begin{array}{ll}
\varepsilon ^2 m_{2,7}&\text{if $n=2$,}\\
\varepsilon ^2 m_{2,7}^2 + \varepsilon\sum _{k=1}^{n-2}m_{2,4k+7}d_{k+2,2} & \text{if $n>2$.}\end{array}\right. \]

In the case $n=2$, $m_{2,7}=0$ because $1-\varepsilon m_{2,7}$ is invertible. Therefore, we have:
\[ M=\left(\begin{array}{cccccccc}
0 & 0 & 0 & 0 & 0 & 0 & 0 & 0 \\
m_{2,1} & 0 & m_{2,3} & 0 & m_{2,5} & 0 & 0 & 0 \\
m_{3,1} & 0 & 0 & 0 & m_{3,5} & 0 & 0 & 0 \\
m_{4,1} & m_{3,1} & m_{4,3} & 0 & m_{4,5} & m_{3,5} & m_{4,7} & 0 \\
\varepsilon m_{2,3} & 0 & 0 & 0 & 0 & 0 & 0 & 0 \\
m_{6,1} & \varepsilon m_{2,3} & 0 & 0 & m_{6,5} & 0 & -m_{2,3} & 0 \\
m_{7,1} & 0 & \varepsilon m_{2,3} & 0 & m_{7,5} & 0 & 0 & 0 \\
m_{8,1} & m_{7,1} & m_{8,3} & \varepsilon m_{2,3} & m_{8,5} & m_{7,5} & m_{8,7} & 0 \end{array}\right) \]
By $M=M^2$, all elements of $M$ must be {$0$.}
 
In the other case, first we prove that the $(4k-2)$-th row of $M$ is zero for all $k=1,2,\ldots ,n$ by induction on $k$.  
By comparing the $(2,4s-1)$-entries of $M$ and $M^2$, the following equations hold:
\begin{equation}\label{2,4s-1}
 m_{2, 4s-1} = \sum_{l=1}^{n-1} m_{2, 4l+3}d_{l+1,s}, \quad s=1,2,\ldots ,n. 
 \end{equation}
Since the first row of $M$ is zero, each $d_{l+1,s}$ of the right hand side of (\ref{2,4s-1}) belongs to  $\varepsilon \mathcal{O}$ and {so does} $m_{2,4s-1}$ for all $s=1,2,\ldots ,n$. Thus, for $s=1,2,\ldots ,n$, the element $m_{2,4s-1}$ {lies in} $\varepsilon^t\mathcal{O}$ for all $t>0$. It implies that $m_{2,4s-1}=0$ for all $s=1,2,\ldots ,n$. Then, the $(2,4s-3)$-entries of $M$ and $M^2$ yield
\[ m_{2,4s-3} = \sum _{l=1}^{n-2} m_{2, 4l+5}d_{l+2,l}, \quad s=1,2,\ldots , n. \] 
As each $d_{l+2,l}$ belongs to $\varepsilon \mathcal{O}$, so is $m_{2,4s-3}$ for all $s=1,2,\ldots ,n$. It implies that the element $m_{2,4s-3}$ {lies in} $\varepsilon ^t\mathcal{O}$ for $t>0$, and hence the second row of $M$ is zero.

Assume that the statement holds for $2 \leq t \leq k-1$, {we will show that the statement for $k$ holds.} Then, by the induction hypothesis, we have
\[ m_{4k-2, 4s-1} = \sum_{l=1}^{n-1} m_{4k-2, 4l+3}d_{l+1,s}, \quad s=1,2,\ldots ,n. \]
Thus, we obtain $m_{4k-2,4s-1}=0$ and 
\[ m_{4k-2,4s-3} = \sum _{l=1}^{n-2} m_{4k-2, 4l+5}d_{l+2,l}, \quad s=1,2,\ldots , n \] 
by {arguments similar to those in} the proof of the case of $k=1$. It implies that  the $(4k-2)$-th row of $M$ is zero for all $k=2,\ldots ,n$.

Since the first and the $(4k-2)$-th row of $M$ are zero for all $k$, the $(i,j)$-block of $M$ is of the form
 \[ M_{i,j}=\left( \begin{array}{cccc}
0 & 0 & 0 & 0 \\
0 & 0  & 0 & 0 \\
m_{4i-1,4j-3}  & 0 & 0  & 0 \\
m_{4i,4j-3}  & m_{4i-1,4j-3}  & m_{4i,4j-1}  & 0  \end{array}\right). \]
Therefore, we obtain $M=\mathbf{0}_{4n}$ by comparing each entry of $M$ and $M^2$.

Next we assume that  $\overline{m_{1,1}}=\overline{1}$. Then, $\mathbf{1}_{4n}-M$ is an idempotent whose $(1,1)$-entry is belongs to $\varepsilon \mathcal{O}$ and $M=\mathbf{1}_{4n}$ follows. 
We have finished the proof of (1).

\para 
Let $X_{(a,b)}$, $Y_{(a,b)}$, $Y_2$ be square matrices of size $4$ defined by
\[ X_{(a,b)}:=\left( \begin{array}{cccc}
0 & 0 & 0 & 0 \\
a & 0 & 0 & 0 \\
0 & 0 & 0 & 0 \\
0 & 0 & b & 0 \end{array}\right)\quad 
Y_{(a,b)}:=\left( \begin{array}{cccc}
0 & 0 & 0 & 0 \\
0 & 0 & 0 & 0 \\
a & 0 & 0 & 0 \\
0 & b & 0 & 0 \end{array}\right)\quad 
Y_2:=\left( \begin{array}{cccc}
0 & 0 & 0 & 0 \\
1 & 0 & 0 & 0 \\
0 & 0 & 0 & 0 \\
0 & 0 & -1 & 0 \end{array}\right), \]
where $a$, $b\in\{1,\varepsilon\}$.
Then, {the representing matrices of the actions of $X$ and $Y$ on $Z_n^{\infty}$ are of the form:}
\[ X=\left(
\begin{array}{ccccccc}
X_{(\varepsilon,1)}&&&&&&\\
&X_{(1,1)}&&&&\hsymb{0}&\\
&&&\ddots&&&\\
&&&&&X_{(1,1)}&\\
&\hsymb{0}&&&&&X_{(1,\varepsilon)}
\end{array}
\right)\]
\[ Y=\left(
\begin{array}{ccccccc}
Y_{(\varepsilon,1)}&&&&&&\\
Y_2&Y_{(\varepsilon,\varepsilon)}&&&&\hsymb{0}&\\
&&&\ddots&&&\\
&&&&&Y_{(\varepsilon,\varepsilon)}&\\
&\hsymb{0}&&&&Y_2&Y_{(1,\varepsilon)}
\end{array}
\right)\]

\begin{lemma}\label{endring}
The endomorphism ring of $Z_n^{\infty}$ is {a} subset of 
\[ \left\{ (m_{i,j})_{i,j} \in \mathrm{Mat}(4n,4n,\mathcal{O}) \left| 
\begin{array}{ll}
\ m_{i,i}=m_{i+1,i+1}\text{ for all $1\leq i\leq 4n-1$},\\
\begin{array}{lcl} 
m_{i,j}=0\text{ for $i < j$ whenever } (i,j) &\neq &(2,3), (2,5), (4,5)\\
& &\text{$(4,7)$ or $(8,9)$}.
 \end{array}
 \end{array}\right.\right\} \]
\end{lemma}
\begin{proof} 
The proof is straightforward.
\end{proof}

\para \textbf{Proof of (2) in Proposition \ref{indec of Heller}}

Let $M$ be an idempotent of the endomorphism ring of $Z^{\infty}_n$. It follows from Lemma \ref{endring} that $M$ must be either the zero matrix or the identity matrix by comparing all entries of $M$ with those of $M^2$. Therefore, the $A$-lattice $Z^{\infty}_n$ is indecomposable. 

\para \textbf{Proof of (3) in Proposition \ref{indec of Heller}}

For {$n\geq 1$}, we define $\overline{A}$-submodules of $Z_n^{\lambda}\otimes\kappa$ by 
\begin{align*} 
Z(\lambda,n,1):=&\sf{Span}_{\kappa}\{\sf{a}_{i,1}^{\lambda},\ \sf{a}_{i,2}^{\lambda}\ \mid\ i=1,\ldots, n \}, \\
Z(\lambda,n,2):=&\sf{Span}_{\kappa}\{\sf{a}_{i,3}^{\lambda},\ \sf{a}_{i,4}^{\lambda}\ \mid\ i=1,\ldots, n \}. 
\end{align*}
Then, $Z_n^{\lambda}\otimes\kappa$ is decomposed into $Z(\lambda,n,1)\oplus Z(\lambda,n,2)$ as $\overline{A}$-modules. 
Define $\overline{A}$-homomorphisms $f^{\lambda,n}_1:M(\lambda)_n\to Z(\lambda,n,1)$ and $f^{\lambda,n}_2:M(-\lambda)_n\to Z(\lambda,n,2)$ by
\[ f^{\lambda,n}_1(u_i)=\sf{a}^{\lambda}_{i,1},\quad  f^{\lambda,n}_1(v_i)=\sf{a}^{\lambda}_{i,2},\quad  f^{\lambda,n}_2(u_i)=(-1)^{i+1}\sf{a}^{\lambda}_{i,3},\text{ and }  f^{\lambda,n}_2(v_i)=(-1)^{i+1}\sf{a}^{\lambda}_{i,4}. \]
As these morphisms are isomorphisms, we have the assertion. 

\para \textbf{Proof of (4) in Proposition \ref{indec of Heller}}

For {$n\geq 1$}, we put 
\begin{align*} 
Z(\infty,n,1):=&\sf{Span}_{\kappa}\{\sf{b}_{i,1},\ \sf{b}_{j,2}, \sf{b}_{n,3}\ \mid\ i=1,\ldots, n,\ j=2,\ldots, n \}, \\
Z(\infty,n,2):=&\sf{Span}_{\kappa}\{\sf{b}_{1,2},\ \sf{b}_{i,3},\ \sf{b}_{j,4} \mid
i=1,\ldots, n-1,\ j=1,\ldots, n\}.
\end{align*}
Then, one can show that $Z(\infty,n,1)\simeq Z(\infty,n,2)\simeq M(\infty)_n$.


\begin{proposition}\label{notuse} For $\lambda\in\mathbb{P}^1(\kappa)$ and {$n\geq 1$}, the following statements hold.
\begin{enumerate}[(1)]
\item If $\lambda\neq \infty$, there exists an isomorphism $\tau Z_n^{\lambda}\simeq Z_n^{-\lambda}$.
\item If $\lambda=\infty$, there exists an isomorphism $\tau Z_n^{\infty}\simeq Z_n^{\infty}$.
\end{enumerate}
\end{proposition}
\begin{proof}
(1)  The map $\pi_{n,\lambda}$ defined by
\[ \begin{array}{ccccl}
\pi _{n,\lambda} & : & A^{2n} & \longrightarrow & \quad Z_n^{\lambda} \\
 & & e_i & \longmapsto & \left\{\begin{array}{ll}
                                                 \sf{a}_{k,1}^{\lambda} & \text{if $i=2k-1$, $k=1,2,\ldots ,n$,} \\
                                                 \sf{a}_{k,3}^{\lambda} & \text{if $i=2k$, $k=1,2,\ldots ,n$} \\
                                              \end{array}\right. \end{array} \]
is the projective cover of $Z_n^{\lambda}$ as an $A$-module. Its kernel $\tau Z_n^{\lambda}$ is given by 
\begin{align*}
 &\mathcal{O}(\varepsilon e_2-Ye_1+\lambda Xe_1) \oplus \mathcal{O} (\varepsilon Xe_2-XYe_1) \oplus \mathcal{O} (Ye_2+\lambda Xe_2) \oplus \mathcal{O}XY e_2 \\
&\bigoplus_{k=2}^{n}\bigg(\mathcal{O}(-1)^{k-1}(\varepsilon e_{2k}-Ye_{2k-1}+\lambda Xe_{2k-1}+Xe_{2k-3}) \oplus \mathcal{O} (-1)^{k-1}(\varepsilon Xe_{2k}-XYe_{2k-1})\\
& \hspace{5cm}\oplus \mathcal{O} (-1)^{k-1}(Ye_{2k}+\lambda Xe_{2k}+Xe_{2k-2}) \oplus \mathcal{O}(-1)^{k-1}XY e_{2k}\bigg).\end{align*}  
Then, the actions $X$ and $Y$ on $\tau Z_n^{\lambda}$ coincide with those on $Z_n^{-\lambda}$.

(2) We define an $A$-module homomorphism by
\[ \begin{array}{ccccl}
\pi _{n,\infty} & : & A^{2n} & \longrightarrow & \quad Z_n^{\infty} \\
 & & e_i & \longmapsto & \left\{\begin{array}{ll}
                                                 \sf{b}_{1,1}& \text{if $i=1,$} \\
                                                 \sf{b}_{1,2} & \text{if $i=2$,} \\
                                                 \sf{b}_{k,3} & \text{if $i=2k+1$,\quad  $k=1,2,\ldots ,n-1,$} \\
                                                 \sf{b}_{k,1} & \text{if $i=2k$,\quad  $k=2,3,\ldots ,n.$} \\
                                                 \end{array}\right. \end{array} \]
Then, the $\pi_{n,\infty}$ is the projective cover of $Z_n^{\infty}$, and an $\mathcal{O}$-basis of the kernel of $\pi_{n,\infty}$ is given as follows. If $n=1$, then the kernel of $\pi_{1,\infty}$ is 
\[ \mathcal{O}(-Xe_{1}+\varepsilon e_{2}) \oplus \mathcal{O}Xe_{2} \oplus\mathcal{O}(-XYe_{1}+\varepsilon Ye_{2})\oplus\mathcal{O}XYe_{2}, \]
and it is isomorphic to $Z_1^{\infty}$. 
If $n=2$, then the kernel of $\pi_{2,\infty}$ is 
\[ \begin{array}{l}
\ \mathcal{O}{(\varepsilon e_{2}-Xe_{1})} \oplus \mathcal{O}Xe_{2} \oplus\mathcal{O}(-Xe_{3}+ Ye_{2})\oplus\mathcal{O}XYe_{2} \\
\oplus \mathcal{O}(-Ye_1+Xe_{4}+\varepsilon e_{3}) \oplus \mathcal{O}(-XYe_{1}+\varepsilon Xe_3) \oplus\mathcal{O}(XYe_{4}+ \varepsilon Ye_{3})\oplus\mathcal{O}XYe_{3}, \\
\end{array} \]
and it is isomorphic to $Z_2^{\infty}$.  Suppose that $n\geq 3$. Then an $\mathcal{O}$-basis of the kernel of $\pi_{n,\infty}$ is given by
\begin{align*}
&\ \mathcal{O}(\varepsilon e_2-Xe_1) \oplus\mathcal{O}Xe_2 \oplus\mathcal{O}(Ye_2-Xe_3)\oplus\mathcal{O}XYe_2 \\
&\oplus\mathcal{O}(\varepsilon e_3+Xe_4-Ye_1) \oplus\mathcal{O}(\varepsilon Xe_3-XYe_1) \oplus\mathcal{O}(Ye_3+Xe_5)\oplus\mathcal{O}XYe_3 \\
&\bigoplus _{k=2} ^{n-2}\bigg( \mathcal{O}(-1)^{k+1}(\varepsilon e_{2k+1}+Xe_{2(k+1)}-Ye_{2k}) \oplus \mathcal{O}(-1)^{k+1}(\varepsilon Xe_{2k+1}-XYe_{2k})\\
&\hspace{5cm}\oplus\mathcal{O}(-1)^{k+1}(Ye_{2k+1}+Xe_{2k+3}) \oplus\mathcal{O}(-1)^{k+1}XYe_{2k+1}\bigg) \\
&\oplus\mathcal{O}(-1)^{n}(\varepsilon e_{2n-1}+Xe_{2n}-Ye_{2(n-1)}) \oplus\mathcal{O}(-1)^{n}(\varepsilon Xe_{2n-1}-XYe_{2(n-1)})\\
&\hspace{5cm}\oplus \mathcal{O}(-1)^{n}(\varepsilon Ye_{2n-1}+XYe_{2n}) \oplus\mathcal{O}(-1)^{n}XYe_{2n-1}. \\
 \end{align*}    
Then, it is easy to check that the actions $X$ and $Y$ on the kernel of $\pi_{n,\infty}$ coincide with those on $Z_n^{\infty}$.
\end{proof}

\section{The case \mbox{\boldmath $\lambda\neq \infty$.}}
\subsection{The almost split sequence ending at \mbox{\boldmath $Z_{n}^{\lambda}$}}

Throughout this subsection, we assume that $\lambda\neq \infty$.

\begin{lemma}
An endomorphism $\rho\in \mathsf{End}_A(Z_n^{\lambda})$ is determined by $\rho(\sf{a}_{1,1}^{\lambda}),\ldots,\rho(\sf{a}_{n,1}^{\lambda})$.
\end{lemma}
\begin{proof}
Let $\rho\in \mathsf{End}_A(Z_n^{\lambda})$. 
For any $k=1,2,\ldots,n$, since $\rho$ is an $A$-module homomorphism, we have $X\rho(\sf{a}_{k,1}^{\lambda})=\rho(X\sf{a}_{k,1}^{\lambda})=\rho(\sf{a}_{k,2}^{\lambda})$ and $\rho(\sf{a}_{k,4}^{\lambda})=\varepsilon^{-1}XY\rho(\sf{a}_{k,1}^{\lambda})$.

{If $n=1$, then $Y\rho(\sf{a_{k,1}^{\lambda}})=\varepsilon\rho(\sf{a_{k,3}^{\lambda}})+\lambda\rho(\sf{a_{k,2}^{\lambda}})$} holds. Thus, $\rho\in \mathsf{End}_A(Z_n^{\lambda})$ is determined by $\rho(\sf{a_{1,1}^{\lambda}})$. 
{Otherwise, we have}
\[ \rho(\sf{a}_{k,3}^{\lambda})=\left\{\begin{array}{ll}
{\varepsilon^{-1}(Y\rho(\sf{a}_{1,1}^{\lambda})-\lambda\rho(\sf{a}_{1,2}^{\lambda}))} & k= 1, \\
{\varepsilon^{-1}(Y\rho(\sf{a}_{k,1}^{\lambda})-\lambda\rho(\sf{a}_{k,2}^{\lambda})-\rho(\sf{a}_{k-1,2}^{\lambda}))} & k\neq 1. \end{array}\right. \]
This completes the proof of the lemma. 
\end{proof}

\begin{lemma}\label{keylemma0} 
Let $\rho\in\mathsf{rad}\mathsf{End}_A(Z_n^{\lambda})$. 
If we write
\[ \rho(\sf{a}_{k,1}^{\lambda})= \sum_{l=1}^{n}c_{l,1}^{(k)}\sf{a}_{l,1}^{\lambda}+ A(k),\quad A(k)\in \sf{Span}_\mathcal{O}\{\sf{a}_{i,j}^{\lambda}\mid j\neq 1\}, \]
where $c_{l,1}^{(k)}\in \mathcal{O}$, then the following statements hold.
\begin{enumerate}[(1)]
 \item {The determinant of $\sf{C}=\begin{pmatrix}
c_{1,1}^{(1)} & \cdots & c_{1,1}^{(n)} \\
\vdots & \ddots & \vdots \\
c_{n,1}^{(1)} & \cdots & c_{n,1}^{(n)} \end{pmatrix}$ is in $\varepsilon\mathcal{O}$.
Therefore, $\sf{C}$ is not invertible.}
 \item $c_{n,1}^{(k)}\in\varepsilon\mathcal{O}$ for all $k=1,2,\ldots, n$. 
 \end{enumerate}
\end{lemma}
\begin{proof}
(1) 
{Suppose that an endomorphism $\rho\in\mathsf{rad}\mathsf{End}_A(Z_n^{\lambda})$ is defined by
\begin{equation}\label{eqq1} 
\rho(\sf{a}_{k,1}^{\lambda})= \sum_{l=1}^{n}c_{l,1}^{(k)}\sf{a}_{l,1}^{\lambda}+ A(k), 
\end{equation}
where $c_{l,1}^{(k)}\in \mathcal{O}$ and $A(k)\in \sf{Span}_\mathcal{O}\{\sf{a}_{i,j}^{\lambda}\mid j\neq 1\}$.}
We show that if the matrix $\sf{C}$ is invertible, then $\rho$ is surjective.
As $XY\sf{a}_{l,1}^{\lambda}=\varepsilon \sf{a}_{l,4}^{\lambda}$ holds for all $l=1,\ldots,n$, we have 
\[ ( \rho(\sf{a}_{1,4}^{\lambda}),\ldots, \rho(\sf{a}_{n,4}^{\lambda}))=(\sf{a}_{1,4}^{\lambda},\ldots ,\sf{a}_{n,4}^{\lambda})\sf{C}. \]
Thus, $\sf{a}_{1,4}^{\lambda},\ldots ,\sf{a}_{n,4}^{\lambda}$ are contained in the image of $\rho$. 
{By multiplying both sides of (\ref{eqq1}) by $X$}, we have 
\[ \rho(\sf{a}_{k,2}^{\lambda})=\sum_{l=1}^{n}c_{l,1}^{(k)}\sf{a}_{l,2}^{\lambda}+XA(k). \]
Since $XA(k)$ belongs to $\sf{Span}_\O\{\sf{a}_{l,4}^{\lambda}\mid l=1,\ldots,n\}$ for each {$k=1,2,\ldots,n$}, there exists $x(k)\in \sf{Span}_{\O}\{\sf{a}_{l,4}^{\lambda}\mid l=1,\ldots,n\}$ such that $\rho(x(k))=XA(k)$. 
Hence, we have
\[ ( \rho(\sf{a}_{1,2}^{\lambda}-x(1)),\ldots, \rho(\sf{a}_{n,2}^{\lambda}-x(n)))=(\sf{a}_{1,2}^{\lambda},\ldots ,\sf{a}_{n,2}^{\lambda})\sf{C}. \]
Therefore, $\sf{a}_{1,2}^{\lambda},\ldots ,\sf{a}_{n,2}^{\lambda}$ belong to the image of $\rho$. 
Finally, we show that $\sf{a}_{1,3}^{\lambda},\ldots ,\sf{a}_{n,3}^{\lambda}$ belong to the image of $\rho$. 
{By multiplying both sides of (\ref{eqq1}) by $Y$}, we have
{
\[ \varepsilon \rho(\sf{a}_{k,3}^{\lambda})=\varepsilon\sum_{l=1}^{n}{c_{l,1}^{(k)}}\sf{a}_{l,3}^{\lambda}+\sum_{l=1}^{n}{c^{(k)}_{l,1}}(\lambda\sf{a}_{l,2}^{\lambda}+\sf{a}_{l-1,2}^{\lambda})+YA(k)-\rho(\lambda\sf{a}_{k,2}^\lambda{+}\sf{a}_{k-1,2}^\lambda) \]
Since $YA(k)$ belongs to $\sf{Span}_\O\{\sf{a}_{l,4}^{\lambda}\mid l=1,\ldots,n\}$ for each {$k=1,2,\ldots,n$}, there exists $y(k)\in \sf{Span}_{\O}\{\sf{a}_{l,2}^{\lambda}, \sf{a}_{l,4}^{\lambda}\mid l=1,\ldots,n\}$ such that
\begin{equation}\label{eqq2}
 \varepsilon\rho(\sf{a}_{k,3}^\lambda)=\varepsilon\sum_{l=1}^nc_{l,1}^{(k)}\sf{a}_{l,3}^\lambda+\rho(y(k)).
 \end{equation}
}
Since the restriction of $\rho$ to $\sf{Span}_{\mathcal{O}}\{\sf{a}^{\lambda}_{l,2}, \sf{a}^{\lambda}_{l,4}\mid l=1,\ldots,n\}$ is a bijection from $\sf{Span}_{\mathcal{O}}\{\sf{a}^{\lambda}_{l,2}, \sf{a}^{\lambda}_{l,4}\mid l=1,\ldots,n\}$ to itself, the equation (\ref{eqq2}) implies that {there exists $y'(k)\in Z_n^{\lambda}$ such that $y(k)=\varepsilon y'(k)$. Then, we have
\[ ( \rho(\sf{a}_{1,3}^{\lambda}-y'(1)),\ldots, \rho(\sf{a}_{n,3}^{\lambda}-y'(n)))=(\sf{a}_{1,3}^{\lambda},\ldots ,\sf{a}_{n,3}^{\lambda})\sf{C}. \]
}
This completes the proof of the statement (1).

(2) The statement for $n=1$ is clear by (1). 
In order to prove this statement for $n>1$, we compute ${\rho}(Y\sf{a}_{k,1}^{\lambda}-\lambda X\sf{a}_{k,1}^{\lambda}-X\sf{a}_{k-1,1}^{\lambda})$ in two ways. 
Since $Y\sf{a}_{k,1}^{\lambda}=\varepsilon\sf{a}_{k,3}^{\lambda}+\lambda\sf{a}_{k,2}+\sf{a}_{k-1,2}^{\lambda}$ and $\sf{a}_{k,2}^{\lambda}=X\sf{a}_{k,1}^{\lambda}$, we have
\begin{equation}\label{eqq3}
\rho(Y\sf{a}_{k,1}^{\lambda}-\lambda X\sf{a}_{k,1}^{\lambda}-X\sf{a}_{k-1,1}^{\lambda})=\varepsilon {\rho}(\sf{a}_{k,3}^{\lambda}). 
\end{equation}

{For each $k=1,2,\ldots, n$, let $c^{(k)}_{l,2}$, $c^{(k)}_{l,3}$, $c^{(k)}_{l,4}$ ($l=1,2,\ldots,n$) be elements of $\mathcal{O}$ such that}
\[
 A(k)=\sum_{l=1}^{n}(c_{l,2}^{(k)}\sf{a}^{\lambda}_{l,2}+c_{l,3}^{(k)}\sf{a}^{\lambda}_{l,3}+c_{l,4}^{(k)}\sf{a}^{\lambda}_{l,4}).
 \]
For $k>1$, the left-hand side of (\ref{eqq3}) is
{ 
\begin{align*} 
&\sum_{l=1}^{n-1}(c_{l+1,1}^{(k)}-c_{l,1}^{(k-1)})\sf{a}_{l,2}^{\lambda}-c_{n,1}^{(k-1)}\sf{a}_{n,2}^{\lambda} \\
&+\varepsilon \sum_{l=1}^{n}c_{l,1}^{(k)}\sf{a}_{l,3}^{\lambda}+\sum_{l=1}^{n-1}(\varepsilon c_{l,2}^{(k)}-2\lambda c_{l,3}^{(k)}-c_{l,3}^{(k-1)}-c_{l+1,3}^{(k)})\sf{a}_{l,4}^{\lambda}+(\varepsilon\sf{c_{n,2}^{(k)}}-2\lambda c_{n,3}^{(k)}-c_{n,3}^{(k-1)})\sf{a}_{n,4}^{\lambda}. 
\end{align*}
}
Thus, the coefficients 
{
\begin{equation}\label{Eq}
c_{l+1,1}^{(k)}-c_{l,1}^{(k-1)}\quad (l=1,\ldots,n-1,\ k=1,2,\ldots,n)
\end{equation}
and $c_{n,1}^{(k-1)}$ belong to $\varepsilon\mathcal{O}$. }
It implies that strictly {upper} triangular entries of the matrix $\sf{C}$ belong to $\varepsilon\O$.
{In particular, $c_{k,1}^{(k)}-c_{k-1,1}^{(k-1)}\in\varepsilon\mathcal{O}$ $(k=2,\ldots,n)$.}

On the other hand, by the statement (1), we have
\[ \mathsf{det}(\sf{C})\equiv c_{1,1}^{(1)}\cdots c_{n,1}^{(n)}+\sum_{e\neq \sigma\in S_n}c_{1,1}^{((\sigma(1))}\cdots c_{n,1}^{(\sigma(n))} \equiv 0\quad \mathrm{mod}\ \varepsilon \mathcal{O}, \]
where $S_n$ is the symmetric group of degree $n$ and $e$ is its identity element.
Hence, $c_{k,1}^{(k)}\equiv 0$ modulo $\varepsilon\O$ for some $k$. 
{Therefore, the assertion follows from (\ref{Eq})}.
\end{proof}

For each $n>1$, we define an endomorphism $\Phi_n^{\lambda}:Z_n^{\lambda}\to Z_{n}^{\lambda}$ by
\[ \sf{a}_{k,1}^{\lambda} \longmapsto \left\{\begin{array}{ll}
\sf{a}_{n,4}^{\lambda} &\text{if $k=n$}, \\
0 &\text{otherwise.}\end{array}\right.\]
Recall that the projective cover of $Z_n^{\lambda}$ is given by 
\[ \begin{array}{ccccl}
\pi _{n,\lambda} & : & A^{2n} & \longrightarrow & \quad Z_n^{\lambda} \\
 & & e_i & \longmapsto & \left\{\begin{array}{lll}
                                                 \sf{a}_{k,1}^{\lambda} & \text{if $i=2k-1$}, & k=1,2,\ldots ,n, \\
                                                 \sf{a}_{k,3}^{\lambda} & \text{if $i=2k$}, &  k=1,2,\ldots ,n.\\
                                              \end{array}\right. \end{array} \]
                                           
\begin{lemma}\label{lambda}
Let  $\Phi_n^{\lambda}$ be the endomorphism of $Z_n^{\lambda}$ as above. Then, the following statements hold.
\begin{enumerate}[(1)]
\item $\Phi_n^{\lambda}$ does not factor through $\pi_{n,\lambda}$.
\item For any $\rho\in\mathsf{radEnd}_A(Z_n^{\lambda})$, {$\Phi_n^{\lambda}\circ\rho$} factors through $\pi_{n,\lambda}$.
\end{enumerate}
\end{lemma}
\begin{proof} 
(1) 
Suppose that $\Phi_n^{\lambda}$ factors through the map $\pi_{n,\lambda}$. 
Let ${\psi}:Z_n^{\lambda}\to A^{2n}$ such that $\Phi_n^{\lambda}={\pi_{n,\lambda}\circ\psi}$. 
Put 
\[ {\psi(\sf{a}_{i,1}^{\lambda})=\sum_{k=1}^{2n}\left\{a^{({i})}_{k,1}e_k+a^{({i})}_{k,2}Xe_k+a^{({i})}_{k,3}Ye_k+a^{({i})}_{k,4}XYe_k\right\}.} \]
By comparing coefficients {of $\sf{a}_{s,4}^\lambda$ in $\pi_{n,\lambda}\circ\psi(\sf{a}_{k,1}^{\lambda})$} with those in $\Phi_n^{\lambda}(\sf{a}_{k,1}^{\lambda})$, we have the following equations:
\begin{equation}\label{eq1}
\begin{array}{ll}
\varepsilon a^{{(i)}}_{2s-1,4}+a_{2s,2}^{{(i)}}-\lambda a_{2s,3}^{(i)}-a_{2s+2,3}^{{(i)}}=0 & \text{{$(s=1,2,\ldots,n-1)$}},
\end{array}
\end{equation}
\begin{equation}\label{eq2}
\varepsilon a^{{(i)}}_{2n-1,4}+a_{2n,2}^{{(i)}}-\lambda a_{2n,3}^{(i)}=\left\{\begin{array}{ll}
1 & \text{if ${i=n}$}, \\
0 & \text{otherwise}.
\end{array} \right.
\end{equation}
Since {$X\sf{a}_{i,1}^\lambda=\sf{a}_{i,2}^\lambda$ and} $Y\sf{a}_{i,1}^\lambda=\varepsilon \sf{a}_{i,3}^\lambda+\lambda\sf{a}_{i,2}^\lambda+\sf{a}_{i-1,2}^\lambda$ ($i=1,\ldots, n$), we have
\begin{equation}\label{eq3}
{\varepsilon \psi(\sf{a}_{i,3}^{\lambda})=\sum_{k=1}^{2n}\left\{-(\lambda a_{k,1}^{(i)}{+}a_{k,1}^{(i-1)})Xe_k+a_{k,1}^{(i)}Ye_k+(a_{k,2}^{(i)}-\lambda a_{k,3}^{{(i)}}-a_{k,3}^{{(i-1)}})XYe_k\right\}, }
\end{equation}
where $a_{k,3}^{\sf{(0)}}=0$, and $1\leq i\leq n$. 

In order to obtain a contradiction, we show that $a_{2n,2}^{(n)}-\lambda a_{2n,3}^{(n)}\in\varepsilon\mathcal{O}$. 
{By (\ref{eq3}), we have $(a_{2n,2}^{(i)}-\lambda a_{2n,3}^{{(i)}}-a_{2n,3}^{{(i-1)}}) \in\varepsilon\mathcal{O}$. 
Thus, it is enough to show that $a_{2n,3}^{{(n-1)}}$ is in $\varepsilon\mathcal{O}$.}
The equation (\ref{eq1}) implies that $a_{2n,3}^{{(n-1)}}\in\varepsilon\mathcal{O}$ if and only if $a_{2n-2,2}^{{(n-1)}}-\lambda a_{2n-2,3}^{(n-1)}\in\varepsilon\mathcal{O}$. 
By repeating this procedure, we deduce that the claim is equivalent to $a_{2,2}^{\sf{(1)}}-\lambda a_{2,3}^\sf{(1)}\in\varepsilon\mathcal{O}$.
However, $a_{2,2}^{{(1)}}-\lambda a_{2,3}^{(1)}\in\varepsilon\mathcal{O}$ follows from the equation (\ref{eq3}). 
Now, we obtain
\[ 1=\varepsilon a^{{(n)}}_{2n-1,4}+a_{2n,2}^{{(n)}}-\lambda a_{2n,3}^{(n)}\in\varepsilon\mathcal{O}, \]
a contradiction.

(2) Let $\rho\in\mathsf{radEnd}_A(Z_n^{\lambda})$. 
We put 
\[ \rho(\sf{a}_{k,1}^{\lambda})=\sum_{{i=1}}^{n}({c_{i,1}^{(k)}}\sf{a}_{i,1}^{\lambda}+{c_{i,2}^{(k)}}\sf{a}_{i,2}^{\lambda}+{c_{i,3}^{(k)}}\sf{a}_{i,3}^{\lambda}+{c_{i,4}^{(k)}}\sf{a}_{i,4}^{\lambda}). \]
Lemma \ref{keylemma0} yields that there exists ${d_{n,1}^{(k)}}\in\mathcal{O}$ such that ${\varepsilon d_{n,1}^{(k)}}={c_{n,1}^{(k)}}$ for each ${k}$.
We define an $A$-module homomorphism $\psi:Z_n^{\lambda}\to A^{2n}$ by {$\psi(\sf{a}_{k,1}^{\lambda})= d_{n,1}^{(k)}XYe_{2n-1}$}.
Then, it is easy to check that $\psi$ is well-defined and {$\Phi_n^{\lambda}\circ \rho(\sf{a}_{k,1}^{\lambda})=d_{n,1}^{(k)}\sf{a}_{n,4}^{\lambda}=\pi_{n,\lambda}\circ\psi(\sf{a}_{k,1}^{\lambda})$}.
\end{proof}

Summing up, we have obtained the following proposition.

\begin{proposition}\label{Middle E2}
Consider the following pull-back diagram:
$$\begin{xy}
(0,15)*[o]+{0}="01",(20,15)*[o]+{Z_n^{-\lambda}}="L",(40,15)*[o]+{E_n^{\lambda}}="E", (60,15)*[o]+{Z_n^{\lambda}}="M",(80,15)*[o]+{0}="02",
(0,0)*[o]+{0}="03",(20,0)*[o]+{Z_n^{-\lambda}}="L2",(40,0)*[o]+{A^{2n}}="nP", (60,0)*[o]+{Z_n^{\lambda}}="nM",(80,0)*[o]+{0}="04",
\ar "01";"L"
\ar "L";"E"
\ar "E";"M"
\ar "M";"02"
\ar "03";"L2"
\ar "L2";"nP"
\ar "nP";"nM"_{\pi_{n,\lambda}}
\ar "nM";"04"
\ar @{-}@<0.5mm>"L";"L2"
\ar @{-}@<-0.5mm>"L";"L2"
\ar "E";"nP"
\ar "M";"nM"^{\Phi_n^{\lambda}}
\end{xy}$$
Then, the upper exact sequence is the almost split sequence ending at {$Z_n^{\lambda}$}.
\end{proposition}
\begin{proof}
The statement follows from Proposition \ref{AKM} and Lemma \ref{lambda}.
\end{proof}

\subsection{The middle term of the almost split sequence ending at \mbox{\boldmath $Z_n^{\lambda}$}}

We denote by $E_n^{\lambda}$ the middle term of the almost split sequence ending at $Z_n^{\lambda}$. By Proposition \ref{Middle E2}, the $A$-lattice $E_n^{\lambda}$ is of the form
\[ E_n^{\lambda}=\{ {x+y}\in A^{2n}\oplus Z_n^{\lambda}\mid {x\in A^{2n}, y\in Z_n^\lambda},\ \pi_{n,\lambda}(x)=\Phi_n^{\lambda}(y) \}. \]
Then, an $\mathcal{O}$-basis of the $A$-lattice $E_n^{\lambda}$ is given as follows:
{If $n=1$, then}
{
\begin{align*}
E_1^\lambda = & \mathcal{O}(\varepsilon e_2+\lambda Xe_1-Ye_1)\oplus \mathcal{O}(\varepsilon Xe_2{+\lambda X e_2})\oplus \mathcal{O}(Ye_2-XYe_1)\oplus \mathcal{O}XYe_2 \\
  & \oplus (\sf{a}_{1,1}^\lambda+Xe_2)\oplus \mathcal{O}\sf{a}_{1,2}^\lambda\oplus \sf{a}_{1,3}^\lambda\oplus \sf{a}_{1,4}^\lambda.
\end{align*}
If $n>1$}
{, then
\begin{align*}
E_n^{\lambda} = & \bigoplus _{k=1}^{n}  \bigg(\mathcal{O}(\varepsilon e_{2k}+\lambda Xe_{2k-1}-Ye_{2k-1}+Xe_{2k-3}) \oplus\mathcal{O}(\varepsilon Xe_{2k}-XYe_{2k-1}) \\
&\quad\quad\quad   \oplus\mathcal{O}(Ye_{2k}+\lambda Xe_{2k}+Xe_{2k-2})\oplus\mathcal{O}(XYe_{2k})\bigg) \\
& \bigoplus_{k=1}^{{n-1}}\bigg(\mathcal{O}\sf{a}_{k,1}^{\lambda} \oplus \mathcal{O}\sf{a}_{k,2}^{\lambda} \oplus \mathcal{O}\sf{a}_{k,3}^{\lambda}\oplus \mathcal{O}\sf{a}_{k,4}^{\lambda}\bigg)\\
& \oplus \mathcal{O}(\sf{a}_{n,1}^\sf{\lambda}+Xe_{2n}) \oplus\mathcal{O}\sf{a}_{n,2}^{\lambda}\oplus \mathcal{O}\sf{a}_{n,3}^{\lambda}\oplus\mathcal{O}\sf{a}_{n,4}^{\lambda}, 
\end{align*}
}
{where we understand that $e_{0}=e_{-1}=0$.}
\begin{lemma}\label{E lambda}
The following statements hold.
\begin{enumerate}[(1)]
\item There is an isomorphism $E_n^{\lambda}\otimes\kappa\simeq M(\lambda)_{n-1}\oplus M(\lambda)_{n+1}\oplus M(-\lambda)_{n}^{\oplus 2}$
{, where $M(\lambda)_0=0$.}
\item We have an isomorphism $(\tau E_n^{\lambda})\otimes\kappa\simeq M(-\lambda)_{{n-1}}\oplus M(-\lambda)_{n+1}\oplus M(\lambda)_{n}^{\oplus 2}$
{, where $M(-\lambda)_0=0$.}
\item {For any $n\geq 1$,} $E_n^{\lambda}$ is a non-projective indecomposable $A$-lattice.
\end{enumerate}
\end{lemma}

\begin{proof}
(1) We define $\overline{A}$-submodules of $E_n^{\lambda}\otimes\kappa$ as follows.
\begin{align*} 
&E(\lambda,n)_1:=\sf{Span}_{\kappa}\left\{\left.
\begin{array}{l}
(\varepsilon e_{2k}+\lambda Xe_{2k-1}-Ye_{2k-1}+Xe_{2k-3}), \\
(\varepsilon Xe_{2k}-XYe_{2k-1})
\end{array}
\right|\ \begin{array}{l}
{k=1,\ldots, n}
\end{array}\right\} \\
&E(\lambda,n)_2:=\sf{Span}_{\kappa}\left\{
\begin{array}{l}
\sf{a}_{k,3}^{\lambda},\ \sf{a}_{k,4}^{\lambda}
\end{array}
\left|\ \begin{array}{l}
k=1,\ldots, n\end{array}\right\}\right. \\
&E(\lambda,n)_3:=\sf{Span}_{\kappa}\left\{\left.
\begin{array}{l}
(Ye_2+\lambda Xe_2),\ (XYe_2),\\
{(Ye_{2k}+\lambda Xe_{2k}+Xe_{2k-2}+\sf{a}_{k-1,1}^{\lambda})},\\
{(XYe_{2k}+\sf{a}_{k-1,2}^{\lambda})},\\
{(\sf{a}_{n,1}^\sf{\lambda}+Xe_{2n}),\ \sf{a}_{n,2}^{\lambda}} \\
\end{array} \right|\ \begin{array}{l}
{k=2,\ldots, n}
\end{array}\right\}
\\
&E(\lambda,n)_4:=\sf{Span}_{\kappa}\left\{
\begin{array}{l}
\sf{a}_{k,1}^{\lambda},\ \sf{a}_{k,2}^{\lambda}
\end{array}
\left|\ \begin{array}{l}
k=1,\ldots, n-1\end{array}\right\}\right. 
\end{align*}
{Here, if $n=1$, then we understand that $E(\lambda,1)_4=0$.}
Then, {we observe that} $E_n^{{\lambda}}\otimes\kappa=E(\lambda,n)_1\oplus E(\lambda,n)_2\oplus E(\lambda,n)_3\oplus E(\lambda,n)_4$ and
\[ E(\lambda,n)_1\simeq  E(\lambda,n)_2\simeq M(-\lambda)_n,\quad E(\lambda,n)_3\simeq M(\lambda)_{n+1},\quad E(\lambda,n)_4\simeq M(\lambda)_{n-1}. \]

(2) This follows from Lemmas \ref{proj cover}, \ref{proj1} and the statement (1).

(3) Suppose that $E_n^{\lambda}$ is decomposable. We write $E_n^{\lambda}=E_1\oplus E_2$ with $E_1\neq 0 \neq E_2$ as $A$-lattices. 
Then, the ranks of the $A$-lattices $E_1$ and $E_2$ are divisible by four. The statement (1) implies that $E_1\otimes\kappa\simeq M(-\lambda)_n^{\oplus 2}$ and $E_2\otimes\kappa\simeq M(\lambda)_{n+1}\oplus M(\lambda)_{n-1}$.

{First, we show that $E_2$ is indecomposable. 
If $n=1$, then it is obvious. Assume that $n>1$ and $E_2=E_{2,1}\oplus E_{2,2}$ with $E_{2,1}\neq 0\neq E_{2,2}$ as $A$-lattices.
We may assume that $E_{2,1}\otimes \kappa\simeq M(\lambda)_{n+1}$ and $E_{2,2}\otimes\kappa\simeq M(\lambda)_{n-1}$.
Hence, $E_{2,1}$ and $E_{2,2}$ are indecomposable.
Let
\[ 0\longrightarrow \tau E_{2,2}\longrightarrow W \longrightarrow E_{2,2}\longrightarrow 0 \]
be the almost split sequence ending at $E_{2,2}$. 
Since $E_{2,2}$ is a direct summand of $E_{n}^\lambda$, the middle term $W$ has the Heller lattice $Z_n^{-\lambda}$ as a direct summand.
This contradicts the fact that the rank of $W$ is $4n-4$.}

Let $0\to \tau E_2\to Z_n^{-\lambda}\oplus {W'}\to E_2\to 0$ be the almost split sequence ending at $E_2$. 
By {Lemmas \ref{proj cover} and} \ref{proj1}, we have {$\tau E_2\otimes\kappa \simeq M(-\lambda)_{n+1}\oplus M(-\lambda)_{n-1}$.}
On the other hand, the induced sequence
\[ 0\to (\tau E_2\otimes\kappa) \to (Z_1^{-\lambda}\otimes\kappa)\oplus ({W'}\otimes\kappa)\to (E_2\otimes\kappa)\to 0 \]
splits, which {contradicts Proposition \ref{indec of Heller} (3).}
\end{proof}

\begin{corollary}\label{cor3}
For any $n>0$ and $\lambda\in\kappa$, the Heller component $\mathcal{CH}(Z_n^{\lambda})$ has no loops.
\end{corollary}
\begin{proof}
{By Lemma \ref{E lambda} (3), Heller lattices appear on the boundary of $\mathcal{CH}(Z_n^{\lambda})$.}
\end{proof}

\subsection{The Heller component containing \mbox{\boldmath $Z_n^{\lambda}$}}

In this subsection, we determine the shape of $\mathcal{CH}(Z_n^{\lambda})$. 

\begin{lemma}\label{infinitely many}
Let $\mathcal{C}$ be a component of stable Auslander--Reiten quiver of $A$. Then, $\mathcal{C}$ has infinitely many vertices.
\end{lemma}
\begin{proof}
The assertion follows from \cite[Proposition 1.26]{AKM} and Theorem \ref{nonperiodic Heller}.
\end{proof}

\begin{theorem}\label{main1}
Let $\mathcal{O}$ be a complete discrete valuation ring, $\kappa$ its residue field and $A=\mathcal{O}[X,Y]/(X^{2},Y^{2})$.
Assume that $\kappa$ is algebraically closed, and $\lambda\neq\infty$.
\begin{enumerate}[(1)]
\item If the characteristic of $\kappa$ is $2$, then $\mathcal{CH}(Z_n^{\lambda})\simeq \mathbb{Z}A_{\infty}/\langle \tau\rangle$.
\item If the characteristic of $\kappa$ is not $2$, then 
\[ {
\mathcal{CH}(Z_{n}^{\lambda})\simeq \left\{\begin{array}{ll}
\mathbb{Z}A_{\infty}/\langle \tau \rangle & \text{if $\lambda=0$}, \\
\mathbb{Z}A_{\infty}/\langle \tau^2 \rangle & \text{if $\lambda\neq 0$}.
\end{array}\right. }\] 
\end{enumerate} 
Moreover, any Heller lattice $Z_{n}^{\lambda}$ appears on the boundary of $\mathcal{CH}(Z_{n}^{\lambda})$.
\end{theorem}

\begin{proof}
Lemma \ref{E lambda} implies that every Heller lattice $Z_{n}^{\lambda}$ appears on the boundary of $\mathcal{CH}(Z_{n}^{\lambda})(=\mathcal{CH}(Z_n^{-\lambda}))$. It follows from Proposition \ref{periodic_case} and Lemma \ref{infinitely many} that the tree class $\overline{T}$ of $\mathcal{CH}(Z_n^{\lambda})$ is one of $A_{\infty}$, $B_{\infty}$, $C_{\infty}$, $D_{\infty}$ or $A_{\infty}^{\infty}$. 

Let $F$ be the middle term of the almost split sequence ending at $E_n^{\lambda}$. Then, $F$ is the direct sum of $Z_{n}^{-\lambda}$ and an $A$-lattice $F_n^{\lambda}$. By {Propositions \ref{split}, \ref{indec of Heller} and Lemma \ref{E lambda},} we have 
\[ F_n^{\lambda}\otimes\kappa\simeq M(\lambda)_{n+1}\oplus M(\lambda)_{n-1}\oplus M(-\lambda)_{n+1}\oplus M(-\lambda)_{n-1}\oplus M(\lambda)_n\oplus M(-\lambda)_n. \]
{Note that $F_n^{\lambda}$ has at most two indecomposable direct summands since $\overline{T}$ is one of infinite Dynkin diagrams. 
By indecomposability of $E_n^{\lambda}$, we know that $\overline{T}$ is neither $B_{\infty}$ nor $A_{\infty}^{\infty}$.}

{Suppose that $\overline{T}=C_{\infty}$. Then, it implies from  Propositions \ref{split}, \ref{indec of Heller} and Lemma \ref{E lambda} that $F_n^\lambda=Z_n^\lambda\oplus \widetilde{F}_n^\lambda$, where 
\[   \widetilde{F}_n^{\lambda}\otimes\kappa\simeq M(\lambda)_{n+1}\oplus M(\lambda)_{n-1}\oplus M(-\lambda)_{n+1}\oplus M(-\lambda)_{n-1}. 
\]
On the other hand, the middle term of the almost split sequence ending at $\widetilde{F}_n^{\lambda}$ has $E_n^{-\lambda}$ as a direct summand. This contradicts Proposition \ref{split}.}

{Next, we suppose that $\overline{T}=D_{\infty}$.}
{In this case,} there is an indecomposable direct summand $W$ of $F_n^{\lambda}$ such that the almost split sequence ending at $W$ is of the form $0\to \tau W\to E_n^{-\lambda}\to W\to 0$. 
Then, the induced exact sequence 
\[ 0\to \tau W\otimes\kappa\to E_n^{-\lambda}\otimes\kappa\to W\otimes\kappa\to 0 \]
splits. However, this situation does not occur for any $W$. 

Therefore, $F^{\lambda}_n$ is an indecomposable
$A$-lattice, and $T=A_{\infty}$.
 \end{proof}

\section{The case \mbox{\boldmath $\lambda=\infty$}}
\subsection{The almost split sequence ending at \mbox{\boldmath $Z_{n}^{\infty}$}}

{This subsection studies} the almost split sequence ending at $Z_{n}^{\infty}$. We see that the following lemmas hold {as in the case that $\lambda\neq\infty$.}
 
\begin{lemma} An endomorphism $\rho\in\mathsf{End}_A(Z_n^{\infty})$ is determined by {$\rho(\sf{b}_{1,1}),\ldots,\rho(\sf{b}_{n,1})$.} 
\end{lemma}
\begin{proof} 
{Suppose that $n=1$. 
Since
\[ X\sf{b}_{1,1}=\varepsilon \sf{b}_{1,2}, \quad Y\sf{b}_{1,1}=\sf{b}_{1,3},\quad Y\sf{b}_{1,2}=\sf{b}_{1,4}, \]
any endomorphism $\rho\in\sf{End}_A(Z_1^\infty)$ is determined by $\rho(\sf{b}_{1,1})$.
}

{Next, we suppose that ${n}>1$. Let $\rho$ be an {endomorphism} of $Z_n^{\infty}$.} 
Since $\rho$ is an $A$-module homomorphism, we have $X\rho(\sf{b}_{1,1})=\rho(X\sf{b}_{1,1})=\varepsilon \rho(\sf{b}_{1,2})$, and hence $\rho(\sf{b}_{1,2})=\varepsilon ^{-1} X\rho(\sf{b}_{1,1})$ follows. 
For $k\neq 1$, we  have $X\rho(\sf{b}_{k,1})=\rho(\sf{b}_{k,2})$. 
Next for $1\leq k\leq n-1$, the equation 
\[ Y\rho(\sf{b}_{k,1})= \varepsilon \rho(\sf{b}_{k,3})+\rho(\sf{b}_{k+1,2}) \]
implies that $\rho(\sf{b}_{k,3})={\varepsilon ^{-1}(Y\rho(\sf{b}_{k,1})-\rho(\sf{b}_{k+1,2}))}$, and for $k=n$, we have $\rho(\sf{b}_{n,3})=Y\rho(\sf{b}_{n,1})$. 
Finally, {we have
\[
\rho(\sf{b}_{k,4}) = \left\{\begin{array}{ll}
X\rho(\sf{b}_{k,3}) & \text{if $k\neq n$,} \\
\varepsilon^{-1}X\rho(\sf{b}_{n,3}) & \text{if $k=n$}. 
\end{array}
\right.
\]
Therefore,  $\rho$ is determined by $\rho(\sf{b}_{1,1}),\ldots,\rho(\sf{b}_{n,1})$.}
\end{proof}

\begin{lemma}\label{keylemma} 
Let $\rho\in\mathsf{rad}\mathsf{End}_A(Z_n^{\infty})$. 
If we write
\[ \rho(\sf{b}_{k,1})= \sum_{l=1}^{n}{d_{l,1}^{(k)}}\sf{b}_{l,1}+ B(k), \] 
where ${d_{l,1}^{(k)}}\in \mathcal{O}$ {and $B(k)\in \sf{Span}_\mathcal{O}\{\sf{b}_{i,j}\mid j\neq 1\}$,} then the following statements hold.
\begin{enumerate}[(1)]
 \item {The determinant of $\sf{D}=\begin{pmatrix}
d_{1,1}^{(1)} & \cdots & d_{1,1}^{(n)} \\
\vdots & \ddots & \vdots \\
d_{n,1}^{(1)} & \cdots & d_{n,1}^{(n)} \end{pmatrix}$ is in $\varepsilon\mathcal{O}$.
Therefore, $\sf{D}$ is not invertible.}
 \item ${d_{n,1}^{(k)}}\in\varepsilon\mathcal{O}$ for all ${k=1,2,\ldots, n}$. 
 \end{enumerate}
\end{lemma}

\begin{proof} 
(1) We show that any {an endomorphism $\rho\in\mathsf{End}_A(Z_n^{\infty})$ such that the corresponding matrix $\sf{D}$ is invertible} is surjective. 
As {$XY\sf{b}_{l,1}=\varepsilon \sf{b}_{l,4}$} holds for ${l=1,\ldots, n}$, we have
\[ (\rho(\sf{b}_{1,4}), \ldots , \rho(\sf{b}_{n,4})) = (\sf{b}_{1,4},\ldots , \sf{b}_{n,4})\sf{D}. \]
Hence, $\sf{b}_{1,4},\ldots , \sf{b}_{n,4}$ {belong to} the image of $\rho$.

Assume that $n=1$.
{By multiplying both sides of $\rho(\sf{b}_{1,1})={d_{1,1}^{(1)}}\sf{b}_{1,1}+B(1)$ by $X$, we have}
\[ \varepsilon\rho(\sf{b}_{1,2})=\varepsilon {d_{1,1}^{(1)}}\sf{b}_{1,2}+XB(1). \]
{Since $XB(1)\in \varepsilon \mathcal{O}\sf{b}_{1,4}$, there exists $t\in\sf{Span}_{\mathcal{O}}\{\sf{b}_{1,4}\}$ such that
\[ d_{1,1}^{(1)}\sf{b}_{1,2}=\rho(\sf{b}_{1,2}-t). \]
As $d_{1,1}^{(1)}$ is invertible by our assumption, $\sf{b}_{1,2}$ belongs to the image of $\rho$.
By multiplying both sides of $\rho(\sf{b}_{1,1})={d_{1,1}^{(1)}}\sf{b}_{1,1}+B(1)$ by $Y$, we have
\[ d_{1,1}^{(1)}\sf{b}_{1,3}=\rho(\sf{b}_{1,3}-s), \]
for some $s\in\sf{Span}_\mathcal{O}\{\sf{b}_{1,4}\}$. Thus, $\sf{b}_{1,3}$ belongs to the image of $\rho$, and hence
$\rho$ is surjective. }

Next, we assume that $n>1$. We {recall} that
\[ X\rho(\sf{b}_{k,1})=\left\{\begin{array}{ll}
\varepsilon \rho (\sf{b}_{1,2}) & \text{if ${k=1}$},\\
\rho(\sf{b}_{k,2}) & \text{if ${k\neq 1}$,}\end{array}\right. \quad Y\rho(\sf{b}_{k,1})=\left\{\begin{array}{ll}
\rho(\varepsilon\sf{b}_{k,3}+\sf{b}_{k+1,2}) & \text{if $k\neq n$}, \\
\rho(\sf{b}_{n,3}) & \text{if ${k=n}$}. \end{array}\right. \]
{Thus, we have}
\[ X\left(\sum_{l=1}^{n}{d_{l,1}^{(k)}}\sf{b}_{l,1}+ B({k})\right)=\varepsilon {d_{1,1}^{(k)}}\sf{b}_{1,2}+\sum_{{l=2}}^{{n}}{d_{l,1}^{(k)}}\sf{b}_{l,2}+XB({k}), \]
\[ Y\left(\sum_{l=1}^{n}{d_{l,1}^{(k)}}\sf{b}_{l,1}+ B({k})\right)=\sum_{l=1}^{n-1}{d_{l,1}^{(k)}}(\varepsilon \sf{b}_{l,3}+\sf{b}_{l+1,2})+{d_{n,1}^{(k)}}\sf{b}_{n,3}+YB({k}). \]
{{Also} we observe} that $XB(\sf{k})$ and $YB({k})$ belong to $\sf{Span}_\O\{\sf{b}_{i,4}\mid i=1,\ldots,n\}$. 

{If $k=1$, then} the equality
\[ \varepsilon \rho(\sf{b}_{1,2}) = \varepsilon d_{1,1}^{(1)}{\sf{b}_{1,2}}+\sum_{{l=2}}^n{d_{l,1}^{(1)}}{\sf{b}_{l,2}}+XB(1) \]
implies that {$d_{l,1}^{(1)}\equiv XB(1)\equiv 0$ modulo $\varepsilon \mathcal{O}$ for all $l=2,3,\ldots n$.} 
Thus, there exists $x(1)\in Z_n^{\infty}$ such that $\varepsilon\rho(x(1))=XB(1)$. 
If $k>1$, then, for each $\sf{k}$, there exists $x(k)\in Z_n^{\infty}$ such that $\rho(x(k))=XB(k)$. 
Therefore, it is easy to see that 
\[ (\rho(\sf{b}_{1,2}-x(1)), \ldots , \rho(\sf{b}_{n,2}-x(n))) = (\sf{b}_{1,2},\ldots , \sf{b}_{n,2})\left( \begin{array}{cccc}
{d_{1,1}^{(1)}}& \varepsilon {d_{1,1}^{(2)}} & & \varepsilon {d_{1,1}^{(n)}}\\
\varepsilon ^{-1}{d_{2,1}^{(1)}}& {d_{2,1}^{(2)}} & & {d_{2,1}^{(n)}}\\
\vdots & \vdots & \cdots & \vdots \\
\varepsilon^{-1} {d_{n,1}^{(1)}} & {d_{n,1}^{(2)}} & & {d_{n,1}^{(n)}}\end{array}\right). \]
Since the determinant of the rightmost matrix in the above equation equals to $\det \sf{D}$, each element $\sf{b}_{l,2}$ {$(l=1,2,\ldots ,n)$} belongs to the image of $\rho$. 

{To show $\sf{b}_{1,3}, \sf{b}_{2,3},\ldots ,\sf{b}_{n,3}$ belong to the image of $\rho$, for each $k=1,2,\ldots ,n$, we take $y(k)\in Z_n^{\infty}$ such that
\[ \rho(y(k))=YB(k)+\sum_{l=1}^{n-1}{d_{l,1}^{(k)}}\sf{b}_{l+1,2}. \] 
}
Then, we have the following equations;
{
\begin{align} 
\rho(\sf{b}_{k+1,2}-y(k)) &= -\varepsilon\rho(\sf{b}_{k,3})+\sum_{l=1}^{n-1}\varepsilon {d_{l,1}^{(k)}}\sf{b}_{l,3}+{d_{n,1}^{(k)}}\sf{b}_{n,3} \quad  \text{(if $k\neq n$)}, \label{a} \\
\rho(\sf{b}_{n,3}-y(n)) & =\sum_{l=1}^{n-1}\varepsilon {d_{l,1}^{(n)}}\sf{b}_{l,3}+{d_{n,1}^{(n)}}\sf{b}_{n,3}. \label{b}
\end{align}
For each $k=1,2,\ldots ,n-1$, all cofficients of $\rho({\sf{b}_{k+1,2}}-y(k))$ belong to $\varepsilon\mathcal{O}$ since
$\rho(\sf{b}_{k+1,2}-y(k))\in\sf{Span}_\mathcal{O}\{\sf{b}_{i,2}, \sf{b}_{i,4} \mid i=1,\ldots,n\}$. 
Thus one can define $z(k)\in Z_{n}^{\infty}$ as follows;
\[ \rho(z(k))= \left\{\begin{array}{ll}
\varepsilon^{-1} \rho(\sf{b}_{k+1,2}-y(k)) & \text{if $k\neq n$}, \\
\rho(-y(n))& \text{if $k=n$}. 
\end{array}\right.\]
Then, the equation (\ref{a}) implies that
\[ d_{n,1}^{(k)}\sf{b}_{n,3} = \varepsilon\left(\rho(\sf{b}_{k,3})+\rho(z(k))-\sum_{l=1}^{n-1}d_{l,1}^{(k)}\sf{b}_{l,3}\right). \]
Hence, $d_{n,1}^{(k)}\in\varepsilon\mathcal{O}$ for $k=1,2,\ldots ,n-1$.
Thus, we have {by (\ref{a}) and (\ref{b})}
\[ (\rho(\sf{b}_{1,3}+z(1)),\ldots,\rho(\sf{b}_{n,3}+z(n)))=(\sf{b}_{1,3},\ldots,\sf{b}_{n,3})\left( \begin{array}{cccc}
{d_{1,1}^{(1)}}&  & {d_{1,1}^{(n-1)}} & \varepsilon {d_{1,1}^{(n)}} \\
\vdots & \cdots & \vdots & \vdots \\
{d_{n-1,1}^{(1)}} &  & {d_{n-1,1}^{(n-1)}} & \varepsilon {d_{n-1,1}^{(n)}} \\
\varepsilon ^{-1}{d_{n,1}^{(1)}}& &\varepsilon ^{-1}{d_{n,1}^{(n-1)}}  & {d_{n,1}^{(n)}}\end{array}\right). \]
}
Since the determinant of the rightmost matrix in the above equation equals to $\det \sf{D}$, each element $\sf{b}_{i,3}$ belongs to the image of $\rho$. 
Therefore, the $A$-morphism $\rho$ is surjective.

(2) The statement for $n=1$ is clear by (1). In order to prove this statement for $n>1$, we compute $\rho(Y\sf{b}_{k,1}-X\sf{b}_{k+1,1})$, for $k=1,2,\ldots ,n-1$, in two ways. 
Set $W(k)=YB(k)-XB(k+1)$. 
Since $Y\sf{b}_{k,1}=\varepsilon\sf{b}_{k,3}+\sf{b}_{k+1,2}$ and $\sf{b}_{k,2}=X\sf{b}_{k,1}$, we have
\[  \rho(Y\sf{b}_{k,1}-X\sf{b}_{k+1,1})=\varepsilon \rho(\sf{b}_{k,3}). \]
On the other hand, we have
\[ \rho(Y\sf{b}_{k,1}-X\sf{b}_{k+1,1})=-\varepsilon {d_{1,1}^{(k+1)}}\sf{b}_{1,2}+\sum_{l=2}^{n}({d_{l-1,1}^{(k)}}-{d_{l,1}^{(k+1)}})\sf{b}_{l,2}+\sum_{l=1}^{n-1}\varepsilon {d_{l,1}^{(k)}}\sf{b}_{l,3}+{d_{n,1}^{(k)}}\sf{b}_{n,3}+W({k}). \]
Thus, we have
\[ {{d_{l-1,1}^{(k)}}-\sf{d_{l,1}^{(k+1)}}\equiv {d_{n,1}^{(k)}}\equiv 0\ \mathrm{mod}\ \varepsilon\mathcal{O}\quad {l=2,\ldots,n},\  {k=1,\ldots, n-1}. }\]
This means the strictly {upper} entries of the matrix $\sf{D}$ belong to $\varepsilon \mathcal{O}$. By the statement (1),
\[ \sf{det}(\sf{D})\equiv {d_{1,1}^{(1)}}{d_{2,1}^{(2)}}\cdots {d_{n,1}^{(n)}} + \sum_{e\neq\sigma\in S_n}\mathrm{sgn}(\sigma){d_{1,1}^{(\sigma(1))}}\cdots {d_{n,1}^{(\sigma(n))}}\equiv {d_{1,1}^{(1)}}{d_{2,1}^{(2)}}\cdots {d_{n,1}^{(n)}}\quad \mathrm{mod}\varepsilon \mathcal{O}, \]
where $S_n$ is the symmetric group of degree $n$ and $e$ is its identity element. Now, the claim is clear.
\end{proof}

Recall that the projective cover of $Z_n^{\infty}$ is given by
\[ \begin{array}{ccccl}
\pi _{n,\infty} & : & A^{2n} & \longrightarrow & \quad Z_n^{\infty} \\
 & & e_i & \longmapsto & \left\{\begin{array}{ll}
                                                 \sf{b}_{1,1}& \text{if $i=1,$} \\
                                                 \sf{b}_{1,2} & \text{if $i=2$,} \\
                                                 \sf{b}_{k,3} & \text{if $i=2k+1$, ${k=1,2,\ldots ,n-1}$}, \\
                                                 \sf{b}_{k,1}& \text{if $i=2k$, ${k=2,3,\ldots ,n}$}. \\
                                                 \end{array}\right. \end{array} \]
Now, for each $n\geq 1$, we define an endomorphism $\Phi_n^{\infty}:Z_n^{\infty}\to Z_n^{\infty}$ by 
\[ \sf{b}_{k,1}\longmapsto \left\{\begin{array}{ll}
\sf{b}_{n,4} & \text{if ${k=n}$,} \\
0 & \text{otherwise.} \end{array}\right. \]

First, we construct the almost split sequence ending at $Z_1^{\infty}$ by using $\Phi_1^{\infty}$. 
\begin{lemma}\label{Z(1,inf)}
Let $\Phi_1^{\infty} :Z_{1}^{\infty}\to Z_{1}^{\infty}$ as above. Then, the following statements hold.
\begin{enumerate}[(1)]
\item $\Phi_1^{\infty}$ does not factor through $\pi_{1,\infty}$.
\item For any $\rho\in\mathsf{rad}\mathsf{End}_A(Z_{1}^{\infty})$, {$\Phi_1^{\infty}\circ\rho$} factors through $\pi_{1,\infty}$.
\end{enumerate}
\end{lemma}
\begin{proof} 
(1) Suppose that there exists {$\psi:Z_{1}^{\infty} \to A\oplus A$} such that $\Phi_1^{\infty}={\pi_{1,\infty}\circ\psi}$. 
{Put
\[ \psi(\sf{b}_{1,1})=a_1e_1+a_2Xe_1+a_3Ye_1+a_4XYe_1+b_1e_2+b_2Xe_2+b_3Ye_2+b_4XYe_2,
\]
where $a_i,b_i\in\mathcal{O}$ $(i=1,2,3,4)$.}
Then, we have
\[ \sf{b}_{1,4} = \Phi_1^{\infty}(\sf{b}_{1,1}) = {\pi_{1, \infty}\circ\psi(\sf{b}_{1,1}) = a_1\sf{b}_{1,1}+(\varepsilon a_2+b_1)\sf{b}_{1,2}+a_3\sf{b}_{1,3}+(\varepsilon a_4+b_3)\sf{b}_{1,4}. }\]
{Thus, $a_1=a_3=0$, $\varepsilon a_2+b_1=0$ and $\varepsilon a_4+b_3=1$.}
{On the other hand, by multiplying $X$ to $\psi(\sf{b_{1,1}})$, we have
\[ \varepsilon \psi(\sf{b_{1,2}})=X{\psi}(\sf{b_{1,1}})=-\varepsilon a_2Xe_2+(1-\varepsilon a_4)XYe_2, \]
a contradiction.}

(2) Let $\rho\in \mathsf{rad}\mathsf{End}_A(Z_{1}^{\infty})$. 
We write $\rho(\sf{b}_{1,1})=\alpha \sf{b}_{1,1}+B(1)$, where $\alpha\in\mathcal{O}$ and $B(1)\in \sf{Span}_{\mathcal{O}}\{\sf{b}_{1,2},\sf{b}_{1,3},\sf{b}_{1,4}\}$. 
By Lemma \ref{keylemma}, $\alpha=\varepsilon\alpha'$ for some $\alpha'\in\mathcal{O}$. 
Define an $A$-module homomorphism $\psi:Z_{1}^{\infty}\to A\oplus A$ by $\psi(\sf{b}_{1,1})=\alpha' XYe_1$. 
Since $\pi_{1,\infty}(\alpha ' XYe_1)=\alpha 'XY\sf{b}_{1,1}=\varepsilon \alpha'\sf{b}_{1,4}$, we have ${\Phi_1^{\infty}\circ\rho(\sf{b}_{1,1})}=\alpha \sf{b}_{1,4}={\pi_{1,\infty}\circ\psi(\sf{b}_{1,1})}$.
\end{proof}

From now on, we construct the almost split sequence ending at $Z_n^{\infty}$ for $n\geq 2$.

\begin{lemma}\label{Z(n,inf)}
Let $\Phi_n^{\infty} :Z_{n}^{\infty}\to Z_{n}^{\infty}$ as above. Then, the following statements hold.
\begin{enumerate}[(1)]
\item $\Phi_n^{\infty}$ does not factor through $\pi_{n,\infty}$.
\item For any $\rho\in\mathsf{rad}\mathsf{End}_A(Z_{n}^{\infty})$, {$\Phi_n^{\infty}\circ\rho$} factors through $\pi_{n,\infty}$.
\end{enumerate}
\end{lemma}
\begin{proof} 
(1) Suppose that there exists {$\psi: Z_n^{\infty}\to A^{2n}$ such that $\Phi_n^{\infty}=\pi_{n,\infty}\circ\psi$.} 
We put 
\begin{equation}\label{psi-exp}
{\psi(\sf{b}_{k,1})=\sum_{l=1}^{2n}\left(a_{l,1}^{(k)}e_l+a_{l,2}^{({k})}Xe_l+a_{l,3}^{({k})}Ye_l +a_{l,4}^{({k})}XYe_l\right).} 
\end{equation}
Then, we notice that, for all ${k=1,\ldots, n}$ and $l=1,\ldots ,2n$, $a_{l,1}^{(k)}$ belongs to $\varepsilon\mathcal{O}$ since $XY\sf{b_{k,1}}=\varepsilon\sf{b}_{k,4}$ for all ${k=1,\ldots, n}$. 
By comparing the coefficient of $\sf{b}_{n,4}$ in $\Phi_{n}^{\infty}(\sf{b}_{n,1})$ with that in ${\pi_{n, \infty}\circ\psi}(\sf{b}_{n,1})$,  we have $\varepsilon a_{2n,4}^{{(n)}}-a_{2n-1,3}^{{(n)}}=1.$ 
In order to obtain a contradiction we show that $a_{2n-1,3}^{{(n)}}\in\varepsilon\mathcal{O}$.
For ${s=1,\ldots, n}$ and  ${t=1,\ldots, n-1}$, by comparing the coefficient of $\sf{b}_{t,4}$ in $\Phi_{n}^{\infty}(\sf{b}_{s,1})$ with that in ${\pi_{n, \infty}\circ\psi}(\sf{b}_{s,1})$, we obtain the following equations:
\begin{align}
\varepsilon a_{1,4}^{({s})}+a_{2,3}^{{(s)}}+a_{3,2}^{{(s)}} = 0 &\quad {t=1},  \label{4.4.2}\\ 
-a_{2t-1,3}^{({s})}+\varepsilon a_{2t,4}^{{(s)}}+{a_{2t+1,2}^{({s})}} = 0 &\quad {t>1}. \label{4.4.3}
\end{align}
On the other hand, (\ref{psi-exp}) implies the following equations:
 \begin{align}
\psi(\sf{b}_{s,2}) & =  X\psi(\sf{b}_{s,1})  =  \sum_{l=1}^{2n}\left( a_{l,1}^{({s})}Xe_l+a_{l,3}^{{(s)}}XYe_l\right)  & {s\neq 1} \label{4.4.4} \\
\varepsilon\psi(\sf{b}_{{s,3}})+\psi(\sf{b}_{{s+1,2}})  & =  Y\psi(\sf{b}_{{s,1}})  = \sum_{l=1}^{2n}\left(  a_{l,1}^{({s})}Ye_l+ a_{l,2}^{{(s)}}XYe_l \right) & {s\neq n} \label{4.4.5}
\end{align}
In particular, it follows from {(\ref{4.4.4}) and (\ref{4.4.5}) that the equation 
\[ \varepsilon\psi(\sf{b}_{s,3})=\sum_{l=1}^{2n}\left( {-a_{l,1}^{{(s+1)}}}Xe_l+a_{l,1}^{(s)}Ye_l+\left(a_{l,2}^{(s)}-a_{l,3}^{(s+1)}\right)XYe_l\right) \] 
holds when $s\neq n$. 
This implies that $a_{2n-1,3}^{(n)}\in\varepsilon \mathcal{O}$ if and only if $a_{2n-1,2}^{(n-1)}\in\varepsilon\mathcal{O}$.
By (\ref{4.4.3}),  it is equivalent to $a_{2n-3,3}^{(n-1)}\in\varepsilon \mathcal{O}$.
By repeating this procedure, we deduce that $a_{2n-1,3}^{(n)}\in\varepsilon\mathcal{O}$ if and only if $a_{3,2}^{(1)}\in\varepsilon \mathcal{O}$. 
By using (\ref{4.4.2}), $a_{3,2}^{(1)}\in\varepsilon \mathcal{O}$ if and only if $a^{(1)}_{2,3}\in\varepsilon\mathcal{O}$.
Since
\[ \varepsilon \psi(\sf{b}_{1,2})=X\psi(\sf{b}_{1,1})=\sum_{l=1}^{2n}\left(a_{l,1}^{(1)}Xe_l+a_{l,3}^{(1)}XYe_l\right), \] 
we conclude that $a_{2,3}^{(1)}$ belongs to $\varepsilon\mathcal{O}$. }

(2) Let $\rho\in\mathsf{rad}\mathsf{End}_A(Z_{n}^{\infty})$. We put 
\[ \rho(\sf{b}_{k,1})= \sum_{l=1}^{n}{d_{l,1}^{(k)}}\sf{b}_{l,1}+ B({k}), \]
where $B({k})\in \sf{Span}_\mathcal{O}\{\sf{b}_{i,j} \mid j\neq 1\}$. 
By Lemma \ref{keylemma}, there are ${e_{n,1}^{(k)}}$ such that ${d_{n,1}^{(k)}}=\varepsilon {e_{n,1}^{(k)}}$. 
Define an $A$-module homomorphism {$\psi:Z_n^{\infty}\to A^{2n}$ as follows:  
\[ \psi(\sf{b}_{k,1}) = {e_{n,1}^{(k)}}XYe_{2n}. \]  
}
Then, it is easy to check that {$\Phi_n^{\infty}\circ\rho(\sf{b}_{k,1})={d_{n,1}^{(k)}}\sf{b}_{n,4}=\pi_{n,\infty}\circ\psi(\sf{b}_{k,1})$}.
\end{proof}

Summing up, we obtain the following proposition.

\begin{proposition}\label{Middle E}
Consider the following pull-back diagram:
$$\begin{xy}
(0,15)*[o]+{0}="01",(20,15)*[o]+{Z_n^{\infty}}="L",(40,15)*[o]+{E_n^{\infty}}="E", (60,15)*[o]+{Z_n^{\infty}}="M",(80,15)*[o]+{0}="02",
(0,0)*[o]+{0}="03",(20,0)*[o]+{Z_n^{\infty}}="L2",(40,0)*[o]+{A^{2n}}="nP", (60,0)*[o]+{Z_n^{\infty}}="nM",(80,0)*[o]+{0}="04",
\ar "01";"L"
\ar "L";"E"
\ar "E";"M"
\ar "M";"02"
\ar "03";"L2"
\ar "L2";"nP"
\ar "nP";"nM"_{\pi_{n,\infty}}
\ar "nM";"04"
\ar @{-}@<0.5mm>"L";"L2"
\ar @{-}@<-0.5mm>"L";"L2"
\ar "E";"nP"
\ar "M";"nM"^{\Phi_n^{\infty}}
\end{xy}$$
Then, the upper exact sequence is the almost split sequence ending at $Z_n^{\infty}$.
\end{proposition}
\begin{proof}
The statement follows from Proposition \ref{AKM} and Lemmas \ref{Z(1,inf)} and \ref{Z(n,inf)}.
\end{proof}

\subsection{The middle term of the almost split sequence ending at \mbox{\boldmath $Z_{n}^{\infty}$}}

In this subsection, we study the middle term of the almost split sequence ending at $Z_n^{\infty}$, say $E_n^{\infty}$, and explain some properties of $E_n^{\infty}$.

\begin{lemma}\label{E(1,inf)}
\begin{enumerate}[(1)]
\item An $\mathcal{O}$-basis of $E_1^{\infty}$ is given by
\[\O(\varepsilon e_2-Xe_1)\oplus\O Xe_2\oplus\O (\varepsilon Ye_2-XYe_1) \oplus \O XYe_2 
\oplus\O(\sf{b}_{1,1}+Ye_2) \oplus \O \sf{b}_{1,2}\oplus\O\sf{b}_{1,3}\oplus \O \sf{b}_{1,4}.\]
\item There is an isomorphism $E_1^{\infty}\otimes \kappa\simeq M(\infty)_{1}^{\oplus 2}\oplus M(\infty)_{2}$.
\item We have an isomorphism $(\tau E_1^{\infty})\otimes\kappa\simeq M(\infty)_{1}^{\oplus 2}\oplus M(\infty)_{2}$. 
\item $E_1^{\infty}$ is a non-projective indecomposable $A$-lattice. 
\end{enumerate}
\end{lemma}
\begin{proof} (1) Straightforward. 

(2) We put 
\begin{align*}
& E(\infty,1)_1:= \sf{Span}_\kappa\{(\varepsilon e_2-Xe_1), ({\varepsilon Ye_2}-XYe_1)\}, \\
& E(\infty,1)_2:=\sf{Span}_\kappa\{\sf{b}_{1,2}, \sf{b}_{1,4}\}, \\
& E(\infty,1)_3:=\sf{Span}_\kappa\{(Xe_2), (XYe_2), (\sf{b}_{1,1}+Ye_2),\sf{b_{1,3}} \}.
\end{align*}
Then, it is easy to check that $E(\infty,1)_1\simeq E(\infty,1)_2\simeq M(\infty)_1$ and $E(\infty,1)_3\simeq M(\infty)_2$.

(3) This follows from Lemmas \ref{proj cover}, \ref{proj1} and the statement (2).

(4) Suppose that $E_1^{\infty}$ is decomposable. We write $E_1^{\infty}=E_1\oplus E_2$ as $A$-lattices with $E_1\neq 0\neq E_2$. Then, the ranks of $E_1$ and $E_2$ are divisible by four. Thus, one can assume that $E_1\otimes\kappa\simeq M(\infty)_1^{\oplus 2}$, $E_2\otimes\kappa\simeq M(\infty)_2$, and $E_1$ and $E_2$ are indecomposable. Then, the $A$-lattice $E_2$ is not isomorphic to any Heller lattices by Theorem \ref{nonperiodic Heller} and Proposition \ref{indec of Heller}. Let $0\to\tau E_2 \to Z_1^{\infty} \oplus W \to E_2 \to 0$ be the almost split sequence ending at $E_2$.
By applying $-\otimes\kappa$, the induced sequence 
\[ 0\to\tau E_2\otimes\kappa \to Z_1^{\infty}\otimes\kappa \oplus W\otimes\kappa \to E_2\otimes\kappa \to 0 \]
splits, which {contradicts Proposition \ref{indec of Heller} (4)}. \end{proof}

By the definition of {$E_2^{\infty}$}, we have
\begin{align*}
E_2^{\infty} = &\O(\varepsilon e_2-Xe_1)\oplus\O(Xe_2)\oplus \O (Xe_3-Ye_2) \oplus \O(XYe_2) \\
                     & \oplus \O(\varepsilon e_3+Xe_4-Ye_1)\oplus\O(\varepsilon Xe_3-XY e_1) \oplus\O(\varepsilon Ye_3+XY e_4) \oplus\O(XYe_3)\\
                     & \oplus \O\sf{b}_{1,1}\oplus \O\sf{b}_{1,2} \oplus \O\sf{b}_{1,3}\oplus\O\sf{b}_{1,4}\\ 
                     & \oplus \O(\sf{b}_{2,1}-Ye_3)\oplus \O\sf{b}_{2,2} \oplus \O\sf{b}_{2,3}\oplus\O\sf{b}_{2,4}.\end{align*} 
                     
\begin{lemma}\label{E(2,inf)}
The following statements hold.
\begin{enumerate}[(1)]
\item There is an isomorphism $E_2^{\infty}\otimes\kappa\simeq \oplus M(\infty)_2^{\oplus 2}\oplus M(\infty)_1\oplus M(\infty)_3$.
\item We have an isomorphism $({\tau} E_2^{\infty})\otimes\kappa\simeq M(\infty)_2^{\oplus 2}\oplus M(\infty)_1\oplus M(\infty)_3$. 
\item $E_2^{\infty}$ is a non-projective indecomposable $A$-lattice.
\end{enumerate}
\end{lemma}
\begin{proof} (1) We put 
\begin{align*}
& E(\infty,2)_1:=\sf{Span}_{\kappa}\{(\varepsilon e_2-Xe_1), (\varepsilon Xe_3-XY e_1), (\varepsilon e_3+Xe_4-Ye_1), (\varepsilon Ye_3+XY e_4) \}, \\
& E(\infty,2)_2:=\sf{Span}_{\kappa}\{\sf{b}_{1,2}, \sf{b}_{1,3}, \sf{b}_{1,4}, \sf{b}_{2,4}\}, \\
& E(\infty,2)_3:=\sf{Span}_{\kappa}\{ (Xe_2), (Xe_3-Ye_2-\sf{b}_{1,1}), (XYe_2), (XYe_3-\sf{b}_{2,2}), (\sf{b}_{2,1}-Ye_3), \sf{b}_{2,3} \}\\
& E(\infty,2)_4:=\sf{Span}_{\kappa}\{ \sf{b}_{1,1}, \sf{b}_{2,2} \} 
\end{align*}
Then, it is easy to check that $E(\infty,2)_1\simeq E(\infty,2)_2\simeq M(\infty)_2$, $E(\infty,2)_3\simeq M(\infty)_3$ and $E(\infty,2)_4\simeq M(\infty)_1$.

(2) This follows from Lemmas \ref{proj cover}, \ref{proj1} and the statement (1).

(3) Suppose that $E_2^{\infty}$ is decomposable. We write $E_2^{\infty}\simeq E_1\oplus E_2$ as $A$-lattices with $E_1\neq 0\neq E_2$. Then, we may assume that $E_1\otimes\kappa\simeq M(\infty)_2^{\oplus 2}$ and $E_2\otimes\kappa\simeq M(\infty)_1\oplus M(\infty)_3$.
Note that the $A$-lattice $E_2$ is not isomorphic to any Heller lattices, and it is indecomposable. 
Let $0\to \tau E_2\to Z_2^{\infty}\oplus W \to E_2\to 0$ be the almost split sequence ending at $E_2$. 
{By applying $-\otimes\kappa$, the induced sequence 
\[  0\to \tau E_2\otimes\kappa\to (Z_2^{\infty}\otimes\kappa)\oplus (W\otimes\kappa)\to E_2\otimes\kappa \to 0 \]
splits, which contradicts Proposition \ref{indec of Heller} (4)}. 
\end{proof}

From now on, we assume that $n>2$. 
Then, an $\mathcal{O}$-basis of the $A$-lattice $E_n^{\infty}$ is given as follows:
\begin{align*}
E_n^{\infty} = &\ \mathcal{O}(\varepsilon e_2-Xe_1)\oplus \mathcal{O}(Xe_2)\oplus \mathcal{O}(Ye_2-Xe_3)\oplus \mathcal{O}(XYe_2) \\
& \oplus\mathcal{O}(\varepsilon e_3+X{e_4}-Ye_1) \oplus \mathcal{O}(\varepsilon Xe_3-XYe_1)\oplus\mathcal{O}(Ye_3+Xe_5)\oplus\mathcal{O}(XYe_3) \\
&\bigoplus _{k=1}^{n-3}\bigg(\mathcal{O}(\varepsilon e_{2k+3}+ Xe_{2k+4}-Ye_{2k+2}) \oplus\mathcal{O}(\varepsilon Xe_{2k+3}-XYe_{2k+2})\\
&\quad\quad\quad   \oplus\mathcal{O}(Ye_{2k+3}+Xe_{2k+5})\oplus\mathcal{O}(XYe_{2k+3})\bigg) \\
& \oplus\mathcal{O}(\varepsilon e_{2n-1}+Xe_{2n}-Ye_{2n-2})\oplus\mathcal{O}(\varepsilon Xe_{2n-1}-XYe_{2n-2}) \\
&\oplus\mathcal{O}(\varepsilon Ye_{2n-1}+XYe_{2n}) \oplus\mathcal{O}(XYe_{2n-1}) \\
& \bigoplus_{{k=1}}^{{n-1}}\bigg(\mathcal{O}\sf{b}_{k,1} \oplus \mathcal{O}\sf{b}_{k,2} \oplus {\mathcal{O}}\sf{b}_{k,3}\oplus \mathcal{O}\sf{b}_{k,4}\bigg)\\
& \oplus \mathcal{O}(\sf{b}_{n,1}-Y{e_{2n-1}}) \oplus\mathcal{O}\sf{b}_{n,2}\oplus \mathcal{O}\sf{b}_{n,3}\oplus\mathcal{O}\sf{b}_{n,4} 
\end{align*}
{Here, if $n=3$, then we understand that 
\[ \bigoplus _{k=1}^{0}\bigg(\mathcal{O}(\varepsilon e_{2k+3}+ Xe_{2k+4}-Ye_{2k+2}) \oplus\cdots\oplus\mathcal{O}(XYe_{2k+3})\bigg) =0. \]
}
\begin{lemma}\label{Middle E 1} {For $n> 2$,} the following statements hold.
\begin{enumerate}[(1)]
\item There is an isomorphism $E_n^{\infty}\otimes\kappa\simeq  M(\infty)_n^{\oplus 2}\oplus M(\infty)_{n+1}\oplus M(\infty)_{n-1}$.
\item We have an isomorphism $(\tau E_n^{\infty})\otimes\kappa\simeq M(\infty)_{n}^{\oplus 2}\oplus M(\infty)_{n-1}\oplus M(\infty)_{n+1}$. 
\item $E_n^{\infty}$ is a non-projective indecomposable $A$-lattice.
\end{enumerate}
\end{lemma}
\begin{proof}
(1) 
We define $\overline{A}$-submodules of $E_n^{\infty}\otimes\kappa$ as follows.
\begin{align*} 
&E(\infty,n)_1:=\sf{Span}_{\kappa}\left\{\left.
\begin{array}{l}
(\varepsilon e_2-Xe_1),\ (\varepsilon Xe_3-XYe_1)\\
{(\varepsilon e_3+X{e_4}-Ye_1)} \\
(\varepsilon e_{2k+1}+Xe_{2k+2}-Y{e_{2k}}), \\
(\varepsilon Xe_{2l+3}-XYe_{2l+2}),\\
(\varepsilon Ye_{2n-1}+XYe_{2n}) \end{array}
\right|\ \begin{array}{l}
{k=2,\ldots, n-1},\\
 l=1,\ldots, n-2\end{array}\right\} \\
\\
&E(\infty,n)_2:=\sf{Span}_{\kappa}\left\{
\begin{array}{l}
\sf{b}_{1,2},\ \sf{b}_{k,3},\ \sf{b}_{l,4}
\end{array}
\left|\ \begin{array}{l}
k=1,\ldots, n-1,\\
l=1,\ldots, n\end{array}\right\}\right. \\
&E(\infty,n)_3:=\sf{Span}_{\kappa}\left\{\left.
\begin{array}{l}
Xe_2,\ XYe_2, \\
(Ye_2-Xe_3{-\sf{b}_{1,1}}),\\
(Ye_{2k+1}+Xe_{2k+3}{-\sf{b}_{k+1,1}}),\\
 (XYe_{2l+1}-\sf{b}_{l+1,2}),\\
(\sf{b}_{n,1}-Ye_{2n-1}),\ \sf{b}_{n,3}\\
\end{array} \right|\ \begin{array}{l}
k=1,\ldots, n-2,\\
l=1,\ldots,n-1\\
\end{array}\right\}\\
&E(\infty,n)_4:=\sf{Span}_{\kappa}\left\{
\begin{array}{l}
\sf{b}_{s,1},\ \sf{b}_{t,2}
\end{array} \left|\ 
\begin{array}{l}
s=1,\ldots ,n-1,\\
t=2,\ldots ,n\end{array}\right\}\right.
\end{align*}
Then, it is easy to check that
\begin{align*}
& E_n^{\infty}\otimes\kappa=E(\infty,n)_1\oplus E(\infty,n)_2\oplus E(\infty,n)_3\oplus E(\infty,n)_4,\\
& E(\infty,n)_1\simeq  E(\infty,n)_2\simeq M(\infty)_n,\\
& E(\infty,n)_3\simeq M(\infty)_{n+1}, \\
& E(\infty,n)_4\simeq M(\infty)_{n-1}.
\end{align*}

(2) This follows from Lemmas \ref{proj cover}, \ref{proj1} and the statement (1).

(3)  We can prove the indecomposability of $E_n^{\lambda}$ by using {arguments similar to those in} the proof of the case $\lambda\neq \infty$.
\end{proof}

\begin{corollary}\label{cor2}
$\mathcal{CH}(Z_n^{\infty})\neq \mathcal{CH}(Z_m^{\infty})$ whenever $n\neq m$. Moreover, $\mathcal{CH}(Z_n^{\infty})$ has no loops.
\end{corollary}
\begin{proof}
{By Lemmas \ref{E(1,inf)}, \ref{E(2,inf)} and \ref{Middle E 1} imply that every Heller lattice $Z_{n}^{\infty}$ appears on the boundary of $\mathcal{CH}(Z_{n}^{\infty})$.}
\end{proof}

\subsection{The Heller component containing \mbox{\boldmath $Z_n^{\infty}$}}\label{HCn}

\begin{theorem}\label{main2}
Let $\mathcal{O}$ be a complete discrete valuation ring, $\kappa$ its residue field and $A=\mathcal{O}[X,Y]/(X^{2},Y^{2})$.
Assume that $\kappa$ is algebraically closed. Then, $\mathcal{CH}(Z_{n}^{\infty})\simeq \mathbb{Z}A_{\infty}/\langle \tau \rangle$. Moreover, the Heller lattice $Z_{n}^{\infty}$ appears on the boundary of $\mathcal{CH}(Z_{n}^{\infty})$.
\end{theorem}
\begin{proof}

It follows from Proposition \ref{periodic_case} and Lemma \ref{infinitely many} that the tree class $\overline{T}$ of $\mathcal{CH}(Z_n^{\infty})$ is one of $A_{\infty}$, $B_{\infty}$, $C_{\infty}$, $D_{\infty}$ or $A_{\infty}^{\infty}$.  

Let $F$ be the middle term of the almost split sequence ending at $E_n^{\infty}$. 
Then, $F$ is the direct sum of $Z_{n}^{\infty}$ and an $A$-lattice $F_n^{\infty}$. 
By Propositions \ref{split}, {\ref{indec of Heller} and  Lemmas \ref{E(1,inf)}, \ref{E(2,inf)} and \ref{Middle E 1}}, we have 
\[ F_n^{\infty}\otimes\kappa\simeq M(\infty)_{n+1}^{\oplus 2}\oplus M(\infty)_{n-1}^{\oplus 2}\oplus M(\infty)_n^{\oplus 2}, \]
{where $M(\infty)_0=0$.
Note that $F_n^\infty$ has at most two indecomposable direct summands since $\overline{T}$ is one of infinite Dynkin diagrams.
By indecomposablitity of $E_n^\infty$, we know that $\overline{T}$ is neither $B_\infty$ nor $A_\infty^\infty$.}

{Suppose that $\overline{T}=C_\infty$.
Then, it implies from Propositions \ref{split}, {\ref{indec of Heller} and  Lemmas \ref{E(1,inf)}, \ref{E(2,inf)} and \ref{Middle E 1}} that $F_n^\infty=Z_n^\infty\oplus \widetilde{F}_n^\infty$, where
\[ \widetilde{F}_n^\infty \simeq M(\infty)_{n-1}^{\oplus 2}\oplus M(\infty)_{n+1}^{\oplus 2}. \] 
On the other hand, the middle term of the almost split {sequence} ending at $\widetilde{F}_n^\infty$ has $E_n^\infty$ as a direct summand.
This contradicts Proposition \ref{split}.}

{Next, we suppose that $\overline{T}=D_\infty$. In this case, }
there is an indecomposable direct summand $W$ of $F_n^{\infty}$ such that the almost split sequence ending at $W$ is of the form $0\to \tau W\to E_n^{\infty}\to W\to 0$. 
Then, the induced exact sequence 
\[ 0\to \tau W\otimes\kappa\to E_n^{\infty}\otimes\kappa\to W\otimes\kappa\to 0 \]
splits. However, this situation does not occur for any $W$. 

Therefore, $F^{\infty}_n$ is an indecomposable
$A$-lattice, and $T=A_{\infty}$.
\end{proof}

\bibliographystyle{amsalpha}

\end{document}